\documentclass[12pt,a4paper]{amsart}
\usepackage{comment}
\usepackage[foot]{amsaddr}

\usepackage{mathrsfs}

\usepackage{mathtools}
\makeatletter
\tagsleft@false
\makeatother
\usepackage{float}
\usepackage{amsthm}  
\usepackage{rotating}
\usepackage{amsmath}
\usepackage{mathrsfs}
\usepackage{booktabs}
\usepackage{tikz}
\usepackage{multirow}
\usepackage{amsfonts}
\usepackage{amssymb}
\usepackage{enumerate}
\usepackage{latexsym}
\usepackage{subcaption}
\usepackage{xcolor}
\usepackage[shortlabels]{enumitem}
\usepackage{pdflscape}
\usepackage{adjustbox}
\usepackage{graphicx}
\usepackage{hyperref}
\usepackage[numbers,sort&compress]
{natbib}
\setlist[enumerate,1]{label={\upshape(\roman*)}}

\makeatletter

\newcommand{\Rmnum}[1]
{\expandafter\@slowromancap\romannumeral #1@}
\makeatother

\newtheorem{thm}{Theorem}[section]

\newtheorem{lemma}[thm]{Lemma}

\newtheorem{cor}[thm]{Corollary}

\newtheorem{example}[thm]{Example}
\newtheorem{defin}[thm]{Definition}

\newenvironment{definition}{\begin{defin} \rm}{\end{defin}}
\theoremstyle{definition}
\newtheorem{remark}[thm]{Remark}

\makeatletter
\renewcommand{\p@subfigure}{\thefigure}
\makeatother

\title[Two-disjoint-cycle-cover vertex pancyclicity of split-star networks]{Two-disjoint-cycle-cover vertex pancyclicity of split-star networks}

\date{}
\thanks{$\dagger$These authors contributed equally to this work.
\newline
*Corresponding author}
\author[Liu]{Chang Liu$^{1,\dagger}$}
\author[Niu]{Ruichao Niu$^{2,\dagger}$}
\author[Yang]{Yuefeng Yang$^{1,*}$}

\address{$^1$School of Science,~China University of Geosciences,~Beijing 100083~China}
\email{$^1$cliu@email.cugb.edu.cn}
\address{$^{2}$College of Science, Minzu University of China,~Beijing 100081~China}
\email{$^2$niuruichao@muc.edu.cn}
\email{$^3$yangyf@cugb.edu.cn}

\begin{document}
	
	\begin{abstract}

	Let $r_1$ and $r_2$ be positive integers with $r_1 \le r_2$. A graph $G$ is called $2$-DCC vertex $[r_1,r_2]$-pancyclic if, for any two distinct vertices of $G$ and any integer $\ell \in [r_1,r_2]$, there exist two vertex-disjoint cycles of lengths $\ell$ and $|V(G)|-\ell$, respectively, containing the two vertices separately. In this paper, we investigate the two-disjoint-cycle-cover vertex pancyclicity of the split-star network $S_n^2$. We prove that $S_n^2$ is $2$-DCC vertex $[3,n!/2]$-pancyclic for $n\ge4$.

	\textit{Key words:} split-star networks; vertex-disjoint cycles; two-disjoint-cycle-cover; vertex-pancyclicity.
		
	\end{abstract}
	\maketitle
	\section{Introduction}
	The~$n$-dimensional split-star network introduced by Cheng, Lipman and Park~\cite{CLP01}, denoted by~$S_n^2$.
	Since then, many fundamental properties of~$S_n^2$~have been extensively studied.
	Zhao and Hao~\cite{ZH18}~determined the~$g$-extra connectivity of~$S_n^2$.
	Gu, Hao and Chang~\cite{GHC19,GHC21}~investigated the component connectivity and the edge connectivity of~$S_n^2$.
	Lin, Xu, Zhou and Hsieh analyzed the conditional diagnosability~\cite{LXZH15}~and Lin, Huang, Wang and Xu analyzed the restricted connectivity of~$S_n^2$, and further established its good-neighbor diagnosability~\cite{LHWX20}.
	More recently, Zhao and Wang~\cite{ZW23}~characterized the structure connectivity and substructure connectivity of~$S_n^2$.

   An interconnection network is often modeled as a graph. We follow Bondy and Murty~\cite{BM08}~for notation and terminology not defined in this paper. A {\em graph}~$G=(V,E)$~is an ordered pair consisting of a vertex set~$V$~and an edge set~$E$, where~$V$~is a finite set and~$E\subseteq\{(u,v)\mid u,v\in V,u\neq v\}$. We often write~$G=(V(G),E(G))$~to emphasize the vertex set and the edge set of~$G$. The number of vertices in $G$ is denoted by $|G|$. If~$e=(u,v)\in E(G)$, then~$u$~and~$v$~are called the ends of~$e$. A {\em path}~$P$~is often denoted by~$P=\langle v_1,v_2,\cdots,v_k\rangle$, where~$v_1$~is adjacent only to~$v_2$,~$v_k$~is adjacent only to~$v_{k-1}$, and~$v_i$~is adjacent exactly to~$v_{i-1}$~and~$v_{i+1}$~for all~$1<i<k$. To emphasize the ends of a path, we also write~$P[v_1,v_k]$. Similarly, a {\em cycle} is often denoted by~$C=\langle v_1,v_2,\cdots,v_k,v_1\rangle$, which is obtained from the path~$P[v_1,v_k]$~by adding an edge joining~$v_1$~and~$v_k$. The length of a path or a cycle is defined as the number of its edges. A path or a cycle of length~$k$~is called a~{\em $k$-path} or~{\em $k$-cycle}. A {\em Hamiltonian path} of a graph~$G$~is a spanning path in~$G$, and a Hamiltonian cycle of~$G$~is a spanning cycle in~$G$. A graph~$G$~is {\em Hamiltonian} if it contains a Hamiltonian cycle.

	To explore the possibility of embedding multiple cycles simultaneously, Kung and Chen~\cite{KC15}~extended these concepts.
	A {\em two-disjoint-cycle-cover} ($2$-DCC for short) of a graph $G$ is a pair of vertex-disjoint cycles $C_1$ and $C_2$ in $G$ satisfying $V(C_1)\cup V(C_2)=V(G)$.	A graph~$G$~is said to be {\em 2-DCC~$[r_1,r_2]$-pancyclic} if, for any positive integer~$\ell$~with~$r_1\le \ell\le r_2$,
	there exist a $2$-DCC~$C_1$~and~$C_2$~in~$G$~such that~$C_1$~has length~$\ell$~and~$C_2$~has length~$|V(G)|-\ell$.
	
	A graph~$G$~is said to be~{\em 2-DCC vertex~$[r_1,r_2]$-pancyclic} if, for any integer~$\ell$~with~$r_1\le \ell\le r_2$~and any two distinct vertices
	~$u,v\in V(G)$, there exist two vertex-disjoint cycles~$C_1$~and~$C_2$~in~$G$~such that one of the following holds:
	\medskip
	\begin{itemize}[label=\textbullet,leftmargin=*]
		\item~$C_1$~has length~$\ell$~and contains~$u$, while~$C_2$~has length~$|V(G)|-\ell$~and contains~$v$;
		\item~$C_1$~has length~$|V(G)|-\ell$~and contains~$u$, while~$C_2$~has length~$\ell$~and contains~$v$.
	\end{itemize}
	\medskip
	If~$G$~is~2-DCC vertex~$[r_1,r_2]$-pancyclic, then necessarily~$r_2\le |V(G)|/2$.
	
	Qiao, Meng and Sabir~\cite{QMS22}~investigated the~2-DCC vertex pancyclicity of augmented cubes.
	Niu, Zhou and Xu~\cite{NXL21}~studied the~2-DCC vertex bipancyclicity of bipartite hypercube-like networks.
	Related results for generalized hypercubes and balanced hypercubes have also been progressively developed~\cite{WHC20}. Li, Chen and Lu~\cite{Li25}~investigated the~2-DCC pancyclicity of~$S_n^2$. In this paper, we extend their work by showing that~$S_n^2$~is~2-DCC vertex~$[3,n!/2]$-pancyclic for any~$n\ge4$.

	\section{Preliminaries}
For any positive integers~$m$~and~$n$~with~$m<n$, the set of integers~$\{m,m+1,\ldots,n\}$~is denoted by~$[m,n]$. We use ~$[n]$ instead of~$[1,n]$. For any integer~$n\ge 3$, let~$S_n$~denote the set of all permutations of~$[n]$.

For two subgraphs~$G_1$~and~$G_2$~of~$G$, we use~$G_1\oplus G_2$~to denote the induced subgraph on the vertex set~$V(G_1)\cup V(G_2)$. The graphs~$G$~and~$H$~are {\em isomorphic}, written as~$G\cong H$. An isomorphism from~$G$~to itself is referred to as an {\em automorphism} of~$G$. For subsets $F_1,\ldots,F_s$ of $V(G)$, the notation $G-F_1-\cdots-F_s$ represents the induced subgraph $G[V(G)-(F_1\cup\cdots\cup F_s)]$. When $H$ is a subgraph of $G$, we use $G-H$ as a shorthand for $G-V(H)$.

The graph~$G$~is {\em vertex-transitive} if, for any two vertices~$u$~and~$v$, there exists an automorphism of~$G$~which maps~$u$~to~$v$. A graph~$G$~is said to be~{\em Hamiltonian-connected} if, for every pair of distinct vertices~$u,v\in V(G)$, there exists a Hamiltonian path starting at~$u$~and ending at~$v$.

	\begin{definition}
		Let~$n\in\mathbb{Z}^+$~with~$n\ge3$. The {\em~$n$-dimensional split-star network}~$S_n^2=(V(S_n^2),~E(S_n^2))$~is defined as follows:
		\[V(S_n^2) = S_n,\quad|V(S_n^2)|=n!.\]
		For two distinct vertices~$u=x_1x_2\dots x_n$~and~$v=y_1y_2\dots y_n$~in~$S_n^2$,~$(u,v)\in E(S_n^2)$~if one of the following holds:
		\begin{enumerate}[label=(\roman*)]
			\item~$y_1 = x_2$,~$y_2=x_1$,~$y_j=x_j$~for all~$j\in[3,n]$. Edge $(u,v)$ is called a~$(12)$-edge and the vertex~$v$~is also denoted by~$u\circ(1,2)$;
			\item~$y_1=x_2$,~$y_2 =x_i$,~$y_i=x_1$,~$y_j=x_j$~for all~$j\in[3,n]\setminus\{i\}$. Edge $(u,v)$ is called an~$i$-edge and the vertex~$v$~is also denoted by~$us_i^{-}$;
			\item~$y_1=x_i$,~$y_2 =x_1$,~$y_i=x_2$,~$y_j=x_j$~for all~$j\in[3,n]\setminus\{i\}$. Edge $(u,v)$ is called an~$i$-edge and the vertex~$v$~is also denoted by~$us_i^{+}$; 
		\end{enumerate}
	\end{definition}
\begin{figure}[h!]
	\centering
	\includegraphics[width=\textwidth]{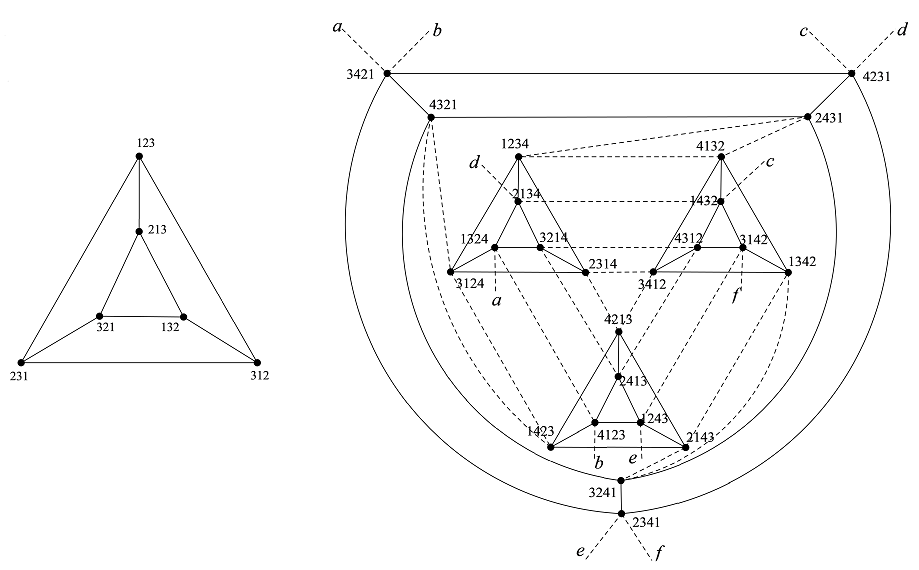}
	\makebox[0.35\textwidth][c]{$S_{4}^{2}$}
	\makebox[0.63\textwidth][c]{$S_{n}^{2}$}
	\caption{Diagrams of the split-star networks.}
	\label{fig:1}
\end{figure}

	\begin{lemma}\label{2}
		~(\cite{LLC21})The~$n$-dimensional split-star network~$S_n^2$~is Hamiltonian-connected for~$n\ge3$.
	\end{lemma}
	
	\begin{lemma}\label{3}
    ~(\cite{Li25})For any edge~$(u,v)\in E(S_n^2)$~with~$n\ge3$,~$S_n^2$~has a Hamiltonian cycle containing~$(u,v)$.
	\end{lemma}

    \begin{lemma}\label{4}
    ~(\cite{CL00})The~$n$-dimensional split-star network~$S_n^2$~is vertex-transitive.
	\end{lemma}

	\begin{lemma}\label{5}
		~(\cite{Li25})~Let~$n\in\mathbb{Z}^+$~with~$n\ge3$. Then:
		\begin{enumerate}[label=(\roman*)]
			\item\label{5-1} For each~$u\in V(S_n^2)$,~$S_n^2-\{u\}$~has a Hamiltonian cycle;
			\item\label{5-2} For each~$(12)$-edge~$(u,v)$, $S_n^2-\{u,v\}$~has a Hamiltonian cycle;
			\item\label{5-3} For each~$i$-edge~$(u,v)$~with~$3\le i\le n$~and~$n\ge4$,~$S_n^2-\{u,v\}$~has a Hamiltonian cycle.
		\end{enumerate}
	\end{lemma}

		For each~$i\in[n]$, denote~$S_n^{2:i}$~as the subgraph of~$S_n^2$~induced by
		\[
		\{x_1x_2\dots x_{n-1}i\mid x_1,\dots,x_{n-1}\text{ are distinct in }[n]\setminus\{i\}\}.
		\]
		Let~$i_1,i_2,\dots,i_k\in[n]$~be distinct integers with~$k\ge2$. Define
		\[
		S_{n,k}^2=S_n^{2:i_1}\oplus S_n^{2:i_2}\oplus\dots\oplus S_n^{2:i_k},\quad V(	S_{n,k}^2)=\bigcup_{t=1}^{k}V(S_n^{2:i_t}).
		\]

	\begin{lemma}\label{6}
		~(\cite{Li25})~In~$S_n^2$~with~$n\ge4$, let~$u=x_1x_2\dots x_n\in V(S_n^2)$~and~$T\subseteq[n]\setminus\{x_1,x_2,x_n\}$~(where~$T$~may be empty). Then there exists a cycle~$C$~in~$S_n^2$~such that
		\[
		V(C)=\left(\bigcup_{t\in T}V(S_n^{2:t})\right)\cup V(S_n^{2:x_1})\cup V(S_n^{2:x_2})\cup\{u\}.
		\]
	\end{lemma}

	\begin{lemma}\label{7}
		Let~$(u,v)\in E(S_n^{2:j})$~with~$n\ge5$~and~$j\in[n]$. Suppose that~$u=x_1x_2\dots x_{n-1}j$~and~$v=y_1y_2\dots y_{n-1}j$. If~$T\subseteq[n]\setminus\{x_1,x_2,j\}$ (where~$T$~may be empty), then there exists a cycle~$C$~in~$S_n^2$~such that
		\[
		V(C)=\left(\bigcup_{t\in T}V(S_n^{2:t})\right)\cup V(S_n^{2:x_1})\cup V(S_n^{2:x_2})\cup\{u,v\}
		\]
		and~$(u,v)\in E(C)$.
	\end{lemma}
	\begin{proof}
If~$v=u\circ(1,2)$~or~$v=u s_k^{-}$~with~$k\in[3,n-1]$, from \cite[Lemma~2.6]{Li25}, then the desired result follows. Hence, we only need to consider the case~$v=u s_k^{+}$.
	Assume that~$T\neq\varnothing$.
	Renumber the integers in~$T$, and write them as~$\{z_1,z_2,\ldots,z_t\}$.
	Denote~$x_2=z_0$~and~$x_1=z_{t+1}$.
	Since~$v=u s_k^{+}$, we have~$y_2=x_1$.
	Let~$u_0=u s_n^{+}$~and~$v_{t+1}=v s_n^{+}$.
	For~$0\le i\le t$, choose~$v_i=z_{i+1}\cdots z_i\in S_{n}^{2:z_i}$~and~$u_{i+1}=v_i s_n^{-}\in S_{n}^{2:z_{i+1}}$.
	According to Lemma \ref{2},~$S_{n}^{2:z_i}$~has a Hamiltonian path~$P_i=[u_i,v_i]$.
	Then
	\[
	  C:=\langle u,u_0,P_0[u_0,v_0],v_0,u_1,P_1[u_1,v_1],v_1,\ldots,P_{t+1}[u_{t+1},v_{t+1}],v_{t+1},v,u\rangle
	\]
	is the required cycle, which is illustrated in Figure \ref{fig:2}.
	
	When~$T=\varnothing$, the required cycle can be obtained by a similar argument.
  \begin{figure}[h!]
  \centering
  \includegraphics[width=\textwidth]{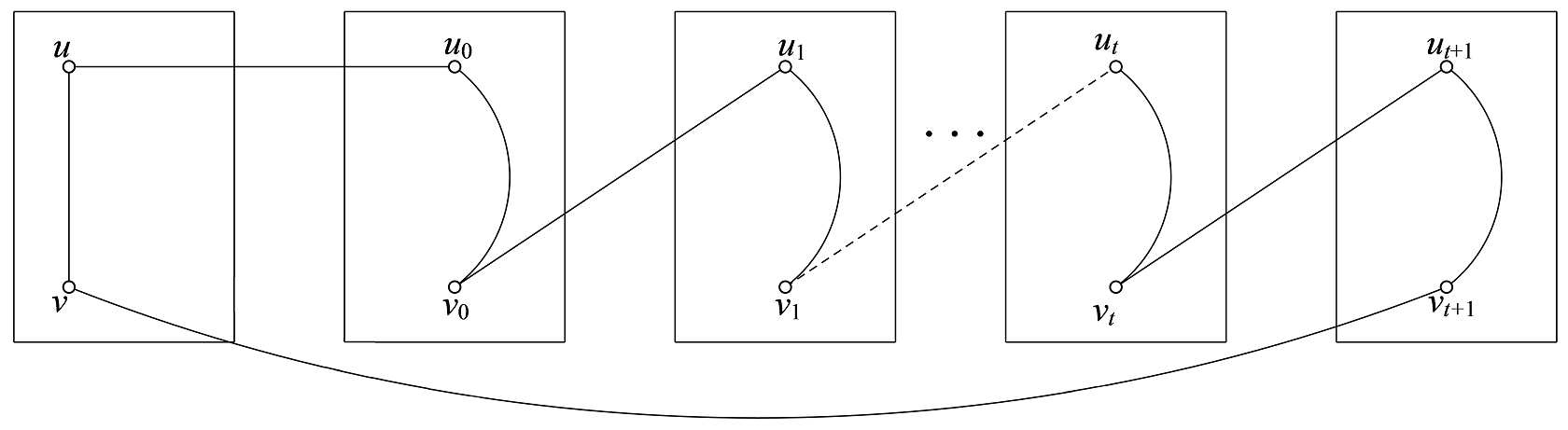}
        \makebox[0.17\textwidth][c]{$S_{n}^{2:j}$}
  \makebox[0.2\textwidth][c]{$S_{n}^{2:z_0}$}
  \makebox[0.2\textwidth][c]{$S_{n}^{2:z_1}$}
  \makebox[0.22\textwidth][c]{$S_{n}^{2:z_{t}}$}
  \makebox[0.16\textwidth][c]{$S_{n}^{2:z_{t+1}}$}
  \caption{Illustration of Lemma \ref{7}.}
  \label{fig:2}
\end{figure}
\end{proof}

	
	\begin{lemma}\label{8}
      		Let~$e=(u,v)$~be an edge in~$S_n^{2:i}$~for some~$i\in[n]$. There exist~$u'\in\{us_n^+,us_n^-\}$~and~$v'\in \{vs_n^+,vs_n^-\}$~such that~$e'=(u',v')\in E(S_n^{2:j})$~for some~$j\in[n]\setminus\{i\}$, which is reffered to as a coupled pair-edge of~$e$. Moreover,~$\langle u,u',v',v,u\rangle$~is a~$4$-cycle.
	\end{lemma}
	
	\begin{proof}
	Let~$u=u_1u_2\cdots u_{n-1}i$, and~$v$~be a neighbor of~$u$~in~$S_n^{2:i}$.
	Without loss of generality, write~$v=v_1v_2\cdots v_{n-1}i$.
	
	\medskip
	\noindent\textbf{Case 1.}~$v=u\circ(1,2)$.

    Note that~$v_1=u_2$~and~$v_2=u_1$.
	Let\[u' = u s_n^{+} =iu_1\cdots u_{n-1}u_2~{\rm and}~v' = v s_n^{-} = u_1i\cdots u_{n-1}u_2\]
	or\[u' = u s_n^{-} = u_2i\cdots u_{n-1}u_1~{\rm and}~v' = v s_n^{+} =iu_2\cdots u_{n-1}u_1.
	\]
	Then the edge~$e=(u,v)$~has a coupled pair-edge~$e'=(u',v')$~in both~$S_n^{2:u_1}$~and~$S_n^{2:u_2}$,
	and~$\langle u,u',v',v,u\rangle$~is a~$4$-cycle in~$S_n^2$.
	
	\medskip
	\noindent\textbf{Case 2.}~$v=us_k^{+}=u_k u_1\cdots i$~for some~$3\le k\le n-1$.
	
	Let
	\[
	u' = u s_n^{-} = u_2i\cdots u_1~{\rm and}~v' = vs_n^{+} =iu_k\cdots u_1.
	\]
	Since~$v'=u's_k^{-}$,~$e=(u,v)$~has a coupled pair-edge~$e'=(u',v')$~in~$S_n^{2:u_1}$, which implies that~$\langle u,u',v',v,u\rangle$~is a~$4$-cycle in~$S_n^2$.
	
	\medskip
	\noindent\textbf{Case 3.}~$v=us_k^{-}=u_2u_k\cdots i$~for some~$3\le k\le n-1$.
	
	Let
	\[
	u' = u s_n^{+} = iu_1\cdots u_2~{\rm and}~v' = vs_n^{-} =u_ki\cdots u_2.
	\]
	Since~$v'=u's_k^{+}$,~$e=(u,v)$~has a coupled pair-edge~$e'=(u',v')$~in~$S_n^{2:u_2}$, which implies that~$\langle u,u',v',v,u\rangle$~is a~$4$-cycle in~$S_n^2$.
	\end{proof}

	\begin{cor}\label{9}
		In~$S_n^2$~with~$n\ge4$, let~$C$~be a cycle in~$S_n^{2:i}$~for some~$i\in[n]$.
		For any edge~$e=uv\in E(C)$, where~$u=u_1u_2\cdots u_{n-1}i$, the following statements hold.
		\begin{itemize}
			\item[(i)] If~$v=u\circ(1,2)$, then~$e$~has a coupled pair-edge in~$S_n^{2:u_1}$~and~$S_n^{2:u_2}$.
			
			\item[(ii)] If~$v=us_k^{-}$,~$3\le k\le n-1$, then~$e$~has a coupled pair-edge in~$S_n^{2:u_2}$.
			
			\item[(iii)] If~$v=us_k^{+}$,~$3\le k\le n-1$, then~$e$~has a coupled pair-edge in~$S_n^{2:u_1}$.
		\end{itemize}
	\end{cor}

	\begin{lemma}\label{10}
        Let~$(u,v)\in E(S_n^2)$~with~$n\ge4$~and~$W_q=\{qw_2w_3\cdots w_n\mid w_2,\ldots,w_n\in[n]\setminus\{q\}\}$. 
        For any~$q\in[n]$, there exist two edges~$(w_q^k,z_q^k)$,~$k\in\{1,2\}$, such that for each~$k$, there exists a Hamiltonian cycle of~$S_n^2$~containing both~$(u,v)$~and~$(w_q^k,z_q^k)$, where~$w_q^k\in W_q$~and~$z_q^k=w_q^ks_n^{+}$.

	\end{lemma}
	\begin{proof}
\textbf{Case 1.}~$u,v\in V(S_n^{2:i})$, where $i\in[n]$.

By Lemma \ref{3},~$(u,v)$~lies on a Hamiltonian cycle~$C_1$~of~$S_n^2$. There exists an edge~$(u_0,v_0)\in E(C_1)$~such that~$u_0\in\{a_1a_2\cdots a_{n-1}i\mid a_j\in [n]\setminus\{i\},~1\leq j\leq n-1\}$ and $(u_0,v_0)\neq(u,v)$. Regardless of whether~$(u_0,v_0)$~is a~$(1,2)$-edge or an~$i$-edge, the vertex~$v_0$~has a neighbor in either~$S_n^{2:a_1}$~or~$S_n^{2:a_2}$,
while~$u_0$~has a neighbor in the other one of these two subgraphs. Consequently,~$u_0$~and~$v_0$~can always be connected to these two subgraphs, respectively. In particular, in the construction below, the specific choice of such a connecting edge is arbitrary, as long as it joins two different subgraphs of~$S_n^2$.

\textbf{Case 1.1.}~$n=4$.

Note that~$u_0=a_1a_2a_3i$. Since the proofs are similar, we may assume~$v_0=u_0\circ(1,2)=a_2a_1a_3i$. The vertex~$v_0$ has a neighbor~$u_2=v_0s_4^{-}=a_1ia_3a_2$~in~$S_4^{2:a_2}$~and~$u_0$~has a neighbor~$v_1=u_0s_4^{-}=a_2ia_3a_1$~in~$S_4^{2:a_1}$.

Suppose~$q\neq i$. Since~$a_1\in [n]\setminus\{i\}$~was arbitrary, we may assume~$q=a_1$. Now we choose two edges~$(w_q^1,z_q^1)$~and~$(w_q^2,z_q^2)$~such that
\[
w_q^{1}=q a_3 i a_2=v_2^1,~z_q^{1}=w_q^{1}s_4^{+}=a_2 q i a_3=u_3^1,
\]
\[
w_q^{2}=q a_2 i a_3=u_3^2,~z_q^{2}=w_q^{2}s_4^{+}=a_3 q i a_2=v_2^2.
\]
Let~$u_1=a_2a_3iq,v_3=qa_2ia_3$. By Lemma \ref{2},~$S_{4}^{2:a_j}$~has a Hamiltonian path~$P_j[u_j,v_j]$, where~$j\in\{1,2,3\}$. Let
	\[
	  C_k:=\langle u_0,P_0[u_0,v_0],v_0,u_2,P_2[u_2,v_2^k],v_2^k,u_3^k,P_3[u_3^k,v_3],v_3,u_1,P_1[u_1,v_1],v_1,u_0\rangle
	\]for~$k\in\{1,2\}$. Since~$(u,v)\in P_0[u_0,v_0]$ and $(w_q^k,z_q^k)\in E(C_k)$, $C_1$ and $C_2$ are the required two cycles, which are illustrated in Figure \ref{fig:3}.
  \begin{figure}[h!]
  \centering
  \includegraphics[width=\textwidth]{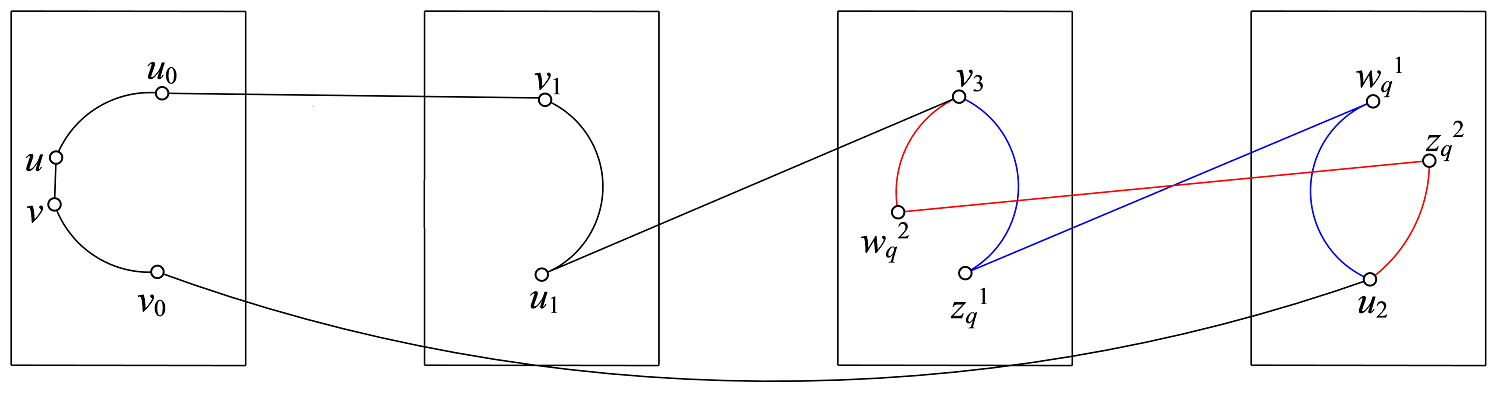}
        \makebox[0.21\textwidth][c]{$S_{n}^{2:i}$}
  \makebox[0.27\textwidth][c]{$S_{n}^{2:a_1}$}
  \makebox[0.27\textwidth][c]{$S_{n}^{2:a_3}$}
  \makebox[0.21\textwidth][c]{$S_{n}^{2:a_2}$}
  \caption{Illustration of Lemma \ref{10}.}
  \label{fig:3}
\end{figure}

Suppose~$q=i$. we choose two edges~$(w_i^1,z_i^1)$~and~$(w_i^2,z_i^2)$~such that
\[
w_i^{1}=i a_3 a_1 a_2=v_2^1,~z_i^{1}=w_i^{1}s_4^{+}=a_2 i a_1 a_3=u_3^1,
\]
\[
w_i^{2}=i a_2 a_1 a_3=u_3^2,~z_i^{2}=w_i^{2}s_4^{+}=a_3 i a_1 a_2=v_2^2.
\]
Let
	\[
	  C_k:=\langle u_0,P_0[u_0,v_0],v_0,u_2,P_2[u_2,v_2^k],v_2^k,u_3^k,P_3[u_3^k,v_3],v_3,u_1,P_1[u_1,v_1],v_1,u_0\rangle
	\]for~$k\in\{1,2\}$. Since~$(u,v)\in P_0[u_0,v_0]$ and $(w_q^k,z_q^k)\in E(C_k)$, $C_1$ and $C_2$ are the required two cycles.

\textbf{Case 1.2.}~$n\ge 5$.

Note that~$u_0=a_1 a_2 \cdots a_{n-1}i$. Since the proofs are similar, we may assume~$v_0=u_0\circ(1,2)=a_2 a_1 \cdots a_{n-1} i$. The vertex~$v_0$ has a neighbor~$u_2=v_0s_n^{-}=a_1ia_3 \cdots a_{n-1}a_2$ in $S_n^{2:a_2}$ and $u_0$~has a neighbor~$v_1=u_0s_n^{-}=a_2ia_3 \cdots a_{n-1}a_1$~in~$S_n^{2:a_1}$.

Suppose~$q\neq i$. Since~$a_1\in [n]\setminus\{i\}$~was arbitrary, we may assume~$q=a_1$. Now we choose two edges~$(w_q^1,z_q^1)$~and~$(w_q^2,z_q^2)$~such that
\[
w_q^{1}=a_1 a_{n-2} \cdots a_3 a_2 i a_{n-1}=u_{n-1}^{1},~z_q^{1}=w_q^{1}s_n^{+}=a_{n-1} a_1 \cdots a_{n-2}=v_{n-2}^{1},
\]
\[
w_q^{2}=a_1 a_{n-2} \cdots a_3 i a_2 a_{n-1}=u_{n-1}^{2},~z_q^{2}=w_q^{2}s_n^{+}=a_{n-1} a_1 \cdots a_{n-2}=v_{n-2}^{2}.
\]
Let~$u_1=a_{n-1}\cdots a_1$~and~$v_{n-1}=u_1s_n^{-}\in V(S_{n}^{2:a_{n-1}})$. For~$2\le t\le n-3$, choose~$v_t=a_{t+1}\cdots a_t\in V(S_{n}^{2:a_t})$~and~$u_{t+1}=v_ts_n^{-}\in V(S_{n}^{2:a_{t+1}})$.
By Lemma \ref{2},~$S_{n}^{2:a_j}$~has a Hamiltonian path~$P_j[u_j,v_j]$, where~$j\in\{1,2,\ldots,n-1\}$.
Let
\[
\begin{aligned}
C_k:= \langle\; &
u_0,P_0[u_0,v_0],v_0,u_2,P_2[u_2,v_2],v_2,\ldots,u_{n-2},P_{n-2}[u_{n-2},v_{n-2}^k],\\
&v_{n-2}^k,u_{n-1}^k,P_{n-1}[u_{n-1}^k,v_{n-1}],v_{n-1},u_1,P_1[u_1,v_1],v_1,u_0\rangle
\end{aligned}
\]for~$k\in\{1,2\}$. Since~$(u,v)\in P_0[u_0,v_0]$ and $(w_q^k,z_q^k)\in E(C_k)$,~$C_1$ and $C_2$ are the required two cycles, which is illustrated in Figure \ref{fig:4}.
  \begin{figure}[h!]
  \centering
  \includegraphics[width=\textwidth]{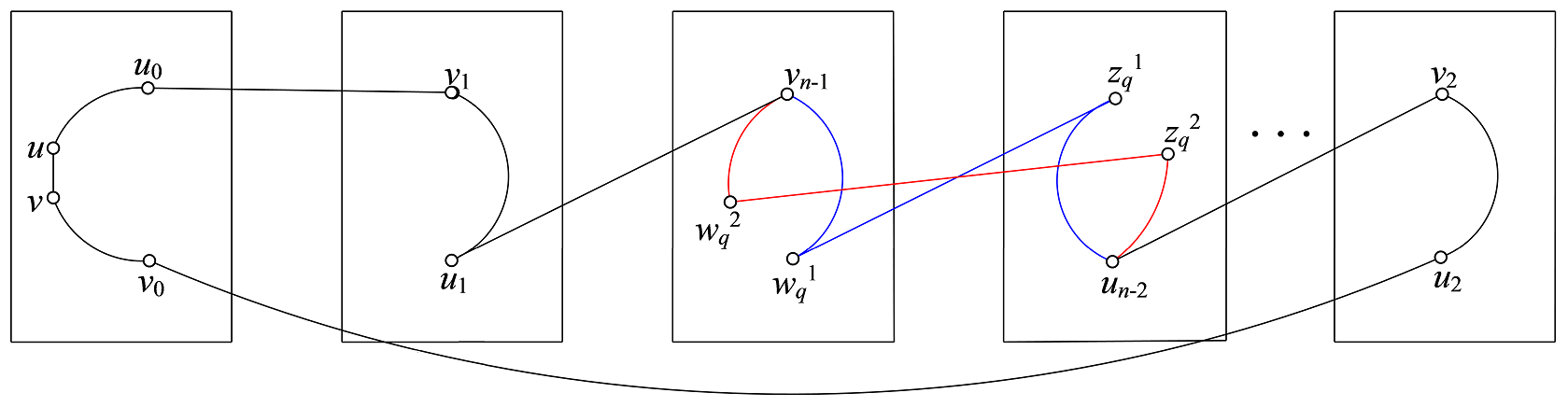}
      \makebox[0.19\textwidth][c]{$S_{n}^{2:i}$}
  \makebox[0.19\textwidth][c]{$S_{n}^{2:a_1}$}
  \makebox[0.2\textwidth][c]{$S_{n}^{2:a_{n-1}}$}
  \makebox[0.19\textwidth][c]{$S_{n}^{2:a_{n-2}}$}
  \makebox[0.19\textwidth][c]{$S_{n}^{2:a_2}$}
  \caption{Illustration of Lemma \ref{10}.}
  \label{fig:4}
\end{figure}

Suppose~$q=i$. we choose two edges~$(w_i^1,z_i^1)$~and~$(w_i^2,z_i^2)$~such that
\[
w_i^{1}=i a_{n-2} \cdots  a_1 a_2 a_{n-1}=u_{n-1}^{1},~z_q^{1}=w_q^{1}s_n^{+}=a_{n-1} i \cdots a_{n-2}=v_{n-2}^{1},
\]
\[
w_i^{2}=i a_{n-2} \cdots  a_2 a_1 a_{n-1}=u_{n-1}^{2},~z_q^{2}=w_q^{2}s_n^{+}=a_{n-1} i  \cdots a_{n-2}=v_{n-2}^{2}.
\]
Let
\[
\begin{aligned}
C_k:= \langle\; &
u_0,P_0[u_0,v_0],v_0,u_2,P_2[u_2,v_2],v_2,\ldots,u_{n-2},P_{n-2}[u_{n-2},v_{n-2}^k],\\
&v_{n-2}^k,u_{n-1}^k,P_{n-1}[u_{n-1}^k,v_{n-1}],v_{n-1},u_1,P_1[u_1,v_1],v_1,u_0\rangle
\end{aligned}
\]for~$k\in\{1,2\}$. Since~$(u,v)\in P_0[u_0,v_0]$ and $(w_i^k,z_i^k)\in E(C_k)$,~$C_1$ and $C_2$ are the required two cycles.

\textbf{Case 2.}~$u \in V(S_n^{2:i})$ and $v\notin V(S_n^{2:i})$, where $i \in [n]$.

Since $u$ and $v$ lie in different subnetworks of $S_n^2$, the vertex $v$ is either $us_n^{-}$ or $us_n^{+}$.
If $v = us_n^{-}$, by treating vertex $u$ as $u_0$ in Case~1, then $v=v_1$, since $v_1=u_0s_n^{-}$ in Case~1. By the construction of $C_k$ in Case~1, for any integer $q \in [n]$, there exist two edges $(w_q^k, z_q^k)$ with $k \in \{1,2\}$ such that for each $k$, there exists a Hamiltonian cycle of $S_n^2$ containing both $(u_0, v_1)$ and $(w_q^k, z_q^k)$. So, there exists a Hamiltonian cycle of $S_n^2$ containing both $(u,v)$ and $(w_q^k, z_q^k)$ in this case.

If $v = u s_n^{+}$, then equivalently $u = v s_n^{-}$. Similar, treat vertex $v$ as $u_0$, then $u=v_1$. There exists a Hamiltonian cycle of $S_n^2$ containing both $(v,u)$ and $(w_q^k, z_q^k)$.

Consequently, the two edges $(u,v)$ and $(w_q^k, z_q^k)$ for each $k \in \{1,2\}$ lie on a Hamiltonian cycle $C_k$ of $S_n^2$. Therefore,~$C_1$ and $C_2$ are the required two cycles.
    \end{proof}

	\begin{lemma}\label{11}
		Let~$n\ge5$~and
		\[
		S_{n,k}^2=S_n^{2:i_1}\oplus S_n^{2:i_2}\oplus\dots\oplus S_n^{2:i_k},\quad i_1,\dots,i_k\in[n],\;1\le k\le n.
		\]
		If $1\le t\le k$, then each edge of~$S_n^{2:i_t}$~is contained in a Hamiltonian cycle of~$S_{n,k}^2$.
	\end{lemma}

\begin{proof}
Since the proofs are similar, we may assume $t=1$. Let $(u,v)\in E(S_{n}^{2:i_1})$.
By Lemma \ref{10}, $S_{n}^{2:i_1}$ contains a Hamiltonian cycle
\[
C_1:=\langle v,P^1[v,z_{i_2}],z_{i_2},w_{i_2},P^1[w_{i_2},u],u,v\rangle.
\] Corollary \ref{9} implies that the edge $(w_{i_2},z_{i_2})$ has a coupled pair-edge
$(w'_{i_2},z'_{i_2})$ in $S_{n}^{2:i_2}$ and $\langle w_{i_2}, w'_{i_2}, z'_{i_2}, z_{i_2}, w_{i_2}\rangle$ is a~$4$-cycle.

By Lemma \ref{10}, $S_{n}^{2:i_2}$ contains a Hamiltonian cycle
\[
C_2:=\langle z_{i_2}^{'},P^2[z_{i_2}^{'},z_{i_3}],z_{i_3},w_{i_3},P^2[w_{i_3},w_{i_2}^{'}],w_{i_2}^{'},z_{i_2}^{'}\rangle.
\] In view of Corollary \ref{9}, the edge $(w_{i_3},z_{i_3})$ has a coupled pair-edge
$(w'_{i_3},z'_{i_3})$ in $S_{n}^{2:i_3}$, and $\langle w_{i_3}, w'_{i_3}, z'_{i_3}, z_{i_3}, w_{i_3}\rangle$ is a~$4$-cycle.

Repeating this process, there is a Hamiltonian cycle\[
C_{k-1}:=\langle z_{i_{k-1}}^{'},P^{k-1}[z_{i_{k-1}}^{'},z_{i_k}],z_{i_k},w_{i_k},P^{k-1}[w_{i_k},w_{i_{k-1}}^{'}],w_{i_{k-1}}^{'},z_{i_{k-1}}^{'}\rangle.
\] Corollary \ref{9} implies that the edge $(w_{i_k},z_{i_k})$ has a coupled pair-edge $(w'_{i_k},z'_{i_k})$ in $S_{n}^{2:i_k}$, and $\langle w_{i_{k-1}}, w'_{i_{k-1}}, z'_{i_{k-1}}, z_{i_{k-1}}, w_{i_{k-1}}\rangle$ is a~$4$-cycle.

According to Lemma \ref{3}, there exists a Hamiltonian cycle\[
C_k:=\langle w_{i_{k}}^{'},z_{i_{k}}^{'},P^{k-1}[z_{i_{k}}^{'},w_{i_{k}}^{'}],w_{i_{k}}^{'} \rangle
\] in $S_{n}^{2:i_k}$. Hence,
\[
\begin{aligned}
  C:=\langle\; &
   u,v,P^1[v,z_{i_2}],z_{i_2},z_{i_2}^{'},P^2[z_{i_2}^{'},z_{i_3}],z_{i_3},\ldots,z_{i_{k-1}}^{'},P^{k-1}[z_{i_{k-1}}^{'},z_{i_k}],z_{i_k},z_{i_k}^{'},\\
   &P^{k}[z_{i_k}^{'},w_{i_k}^{'}],w_{i_k}^{'},w_{i_k},P^{k-1}[w_{i_k},w_{i_{k-1}}^{'}],w_{i_{k-1}}^{'},\ldots,w_{i_3},P^{2}[w_{i_3},w_{i_2}^{'}],\\
   &w_{i_2}^{'},w_{i_2},P^{1}[w_{i_2},u],u \rangle
  \end{aligned}
\] is a Hamiltonian cycle of $S_{n,k}^2$ containing the edge $(u,v)$, which is illustrated in the Figure \ref{fig:5}. Since $(u,v)\in E(S_{n}^{2:i_1})$ was arbitrary, each edge in $S_{n}^{2:i_t}$ lies on a Hamiltonian cycle of $S_{n,k}^2$.
  \begin{figure}[h!]
  \centering
  \includegraphics[width=\textwidth]{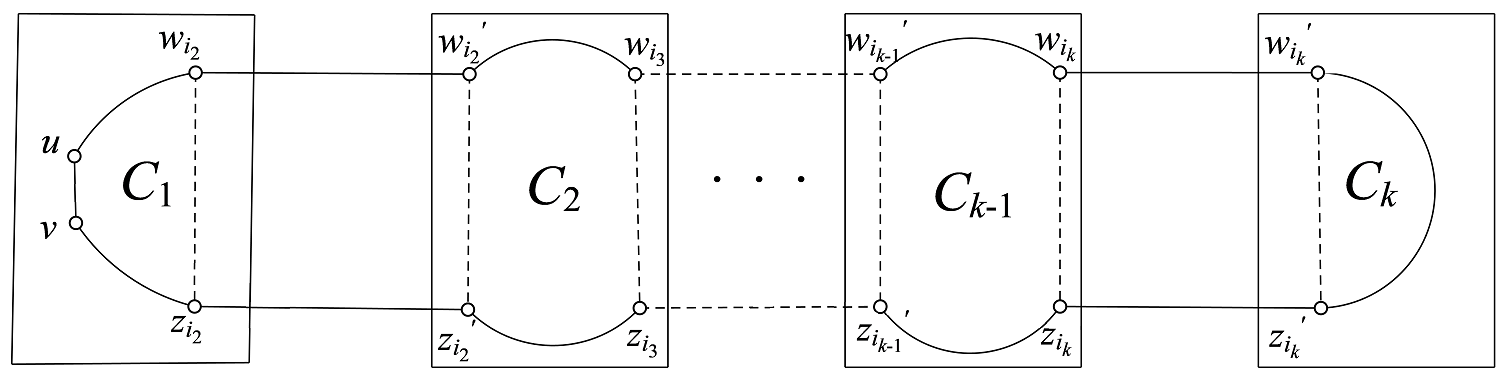}
      \makebox[0.22\textwidth][c]{$S_{n}^{2:i_1}$}
  \makebox[0.24\textwidth][c]{$S_{n}^{2:i_2}$}
  \makebox[0.24\textwidth][c]{$S_{n}^{2:i_{k-1}}$}
  \makebox[0.22\textwidth][c]{$S_{n}^{2:i_k}$}
  \caption{Illustration of Lemma \ref{11}.}
  \label{fig:5}
  \end{figure}
\end{proof}

\begin{lemma}\label{12}
Let $C_1$ and $C_2$ be a two-disjoint-cycle-cover in $S_{n}^{2}$ with $n\ge 4$.
Then there exist two vertices
\[
a_1a_2\cdots a_n\in V(C_1) \quad \text{and} \quad
b_1b_2\cdots b_n\in V(C_2)
\]
such that $a_1,a_2,b_1,b_2$ are pairwise distinct.
\end{lemma}

\begin{proof} Define
\[
\mathscr{A}=\{\{a_1,a_2\}\mid a_1a_2\cdots a_n\in V(C_1)\},\quad\mathscr{B}=\{\{b_1,b_2\}\mid b_1b_2\cdots b_n\in V(C_2)\}.
\] Note that $\mathscr{A}\cup\mathscr{B}=\{\{x_i,x_j\}\mid x_i,x_j\in[n]\}$. It suffices to show that there exist pairwise distinct elements $a_1,a_2,b_1,b_2$ such that $\{a_1,a_2\}\in\mathscr{A}$ and $\{b_1,b_2\}\in \mathscr{B}$. Assume the contrary, namely,  $\{a_1,a_2\},\{a_3,a_4\}\in\mathscr{C}$ for all pairwise distinct elements $a_1,a_2,a_3,a_4\in [n]$ and some $\mathscr{C}\in\{\mathscr{A},\mathscr{B}\}$.

Suppose $\mathscr{A}\cap\mathscr{B}\neq\emptyset$. Without loss of generality, we may assume $\{1,2\}\in\mathscr{A}\cap\mathscr{B}$. Since $A\cup B=\{\{x_i,x_j\}\mid x_i,x_j\in[n]\}$, we have $\{1,2\}\in\mathscr{A}$ and $\{3,4\}\in\mathscr{B}$, or $\{3,4\}\in\mathscr{A}$ and $\{1,2\}\in\mathscr{B}$, a contradiction. Then $\mathscr{A}\cap\mathscr{B}=\emptyset$.


\medskip
\noindent
\textbf{Case 1.} $n=4$.

Without loss of generality, we may assume $\{1,2\},\{3,4\}\in \mathscr{A}$.
Since vertices $12\cdots$ and $34\cdots$ are not adjacent, $C_1$ must contain other vertices of the form $a_1a_2\cdots$, where $\{a_1,a_2\}\notin\{\{1,2\},\{3,4\}\}$. It follows that $\{1,i\},\{2,j\}\in\mathscr{A}$, where $\{i,j\}=\{3,4\}$. Since $\mathscr{B}\neq\emptyset$,~$\{1,j\},\{2,i\}$ must belong to $\mathscr{B}$. Since vertices $1j\cdots$ and $2i\cdots$ are not adjacent, we obtain $\{1,i\},\{2,j\}\in \mathscr{B}$, contrary to the fact that $\mathscr{A}\cap \mathscr{B}\neq\emptyset$.

\medskip
\noindent
\textbf{Case 2.} $n\ge 5$.

Without loss of generality, we may assume $\{1,2\}\in\mathscr{A}$. Since $\{1,2\},\{a,b\}\in\mathscr{A}$ for all distinct elements $a,b\in [3,n]$, there exists $i\in\{1,2\}$ such that $\{i,c\}\in\mathscr{B}$ for some $c\in[3,n]$. Without loss of generality, we may assume $c=3$. Since $n\ge 5$, we have $\{4,5\}\in \mathscr{A}\cup \mathscr{B}$. It follows that $\{1,2\},\{4,5\}\in\mathscr{A}$ or $\{i,3\},\{4,5\}\in\mathscr{B}$, a contradiction.
\end{proof}

	\section{2-DCC vertex pancyclicity of~$S_n^2$}
	
	\begin{lemma}\label{3-1}
		The graph~$S_4^2$~is 2-DCC vertex~$[3,12]$-pancyclic.
	\end{lemma}
    \begin{proof}
      Let $u,v$ be two distinct vertices in $S_4^2$, and $\ell$ be an integer with $3 \le \ell \le 12$. We will construct two disjoint cycles $C_1$ of length $\ell$ and $C_2$ of length $24 - \ell$ such that $u \in V(C_1)$ and $v \in V(C_2)$. By the vertex-transitivity of $S_4^2$, we can assume $u = 1234$, which implies that $u\in S_4^{2:4}$.

        Assume $v \in V(S_4^{2:4})$. When $v = 2134$, for each $\ell \in [3,12]$, we can find two desired cycles as detailed in Table \ref{3-11} of Appendix. These cycles form a 2-DCC containing $u$ and $v$ respectively. If $v \in \{3124,1324,3214,2314\}$, the desired cycles can be constructed as listed in Table \ref{3-12} of Appendix, ensuring $u \in V(C_1)$ and $v \in V(C_2)$.

Assume $v \notin V(S_4^{2:4})$. If $v \in V(S_4^{2:2})$, the required cycles are given in Table \ref{3-13} of Appendix, ensuring $u \in V(C_1)$ and $v \in V(C_2)$. If $v \in V(S_4^{2:1})$ or $v \in V(S_4^{2:3})$, analogous constructions yield the desired 2-DCC containing $u$ and $v$, respectively.
    \end{proof}

	\begin{lemma}\label{3-2}
		Let $n\ge4$. If~$u\in V(S_n^2)$ and $v\in\{u\circ(1,2),us_3^{+}\}$, then each edge of~$S_n^2-\{u,v\}$~is contained in a Hamiltonian cycle of~$S_n^2-\{u,v\}$.
	\end{lemma}

	\begin{proof}
    By the vertex-transitivity of $S_n^2$, we may set $u = n(n-1)\cdots21\in V(S_n^{2:1})$. We prove this lemma by induction on $n$.

    Suppose $n = 4$, then $u = 4321$. For $v = u\circ(1,2) = 3421$, four Hamiltonian cycles in $S_4^2-\{u,v\}$ are illustrated in Fig \ref{fig:99}(a) with colors brown and blue. These four cycles collectively cover all edges of $S_4^2-\{u,v\}$. Similarly, for $v=us_3^{+}=2431$, there exist four Hamiltonian cycles in $S_4^2-\{u,v\}$ that cover all edges of $S_4^2-\{u,v\}$, see Fig \ref{fig:99}(b). Thus, the desired result holds.
\begin{figure}[h!]
    \centering
    \begin{subfigure}[t]{1\textwidth}
        \centering
        \includegraphics[width=\textwidth]{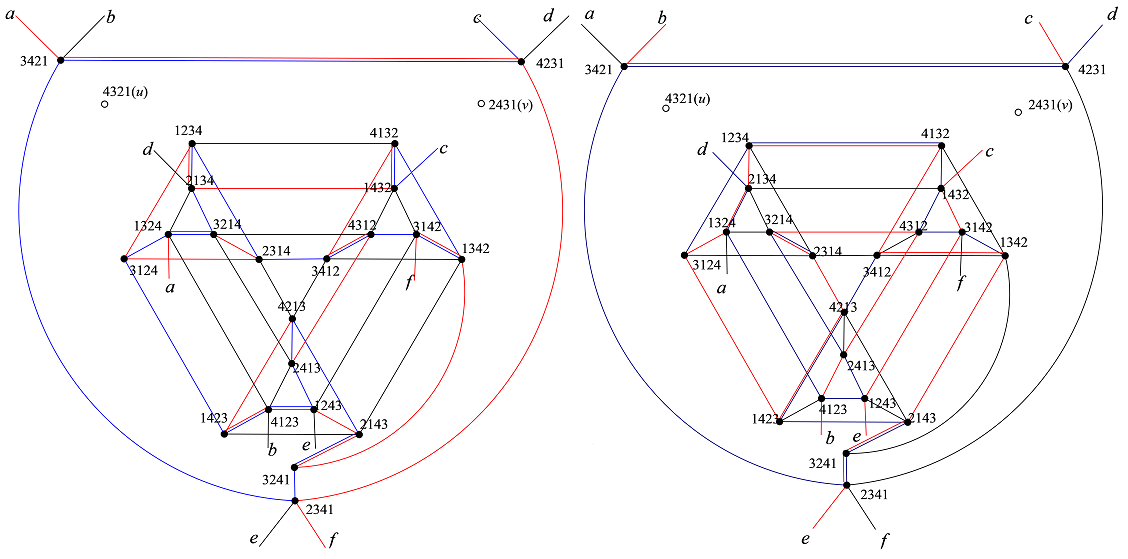}
    \makebox[0.5\textwidth][c]{(a)}
        \label{fig:77}
    \end{subfigure}
    \vspace{0.5cm}
    \begin{subfigure}[t]{1\textwidth}
        \centering
        \includegraphics[width=\textwidth]{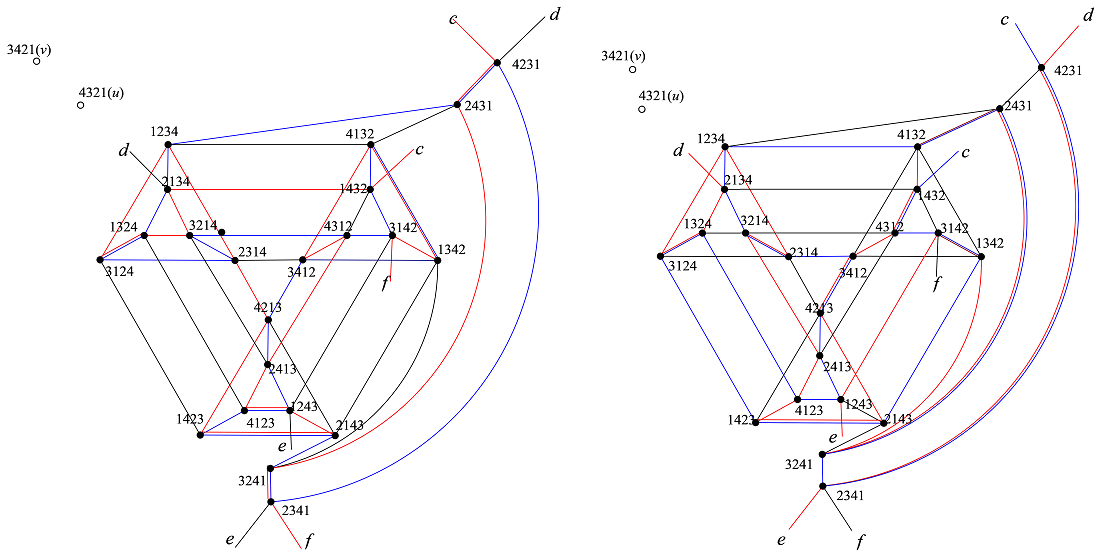}
        \makebox[0.5\textwidth][c]{(b)}
        \label{fig:88}
    \end{subfigure}
    \caption{Hamiltonian cycles of $S_4^2-\{u,v\}$.}
    \label{fig:99}
\end{figure}

    Let $k$ be an integer not less than $4$. Assume the lemma holds for $n\le k$. Now we consider the case $n=k+1$. Since $v=u\circ(1,2)=k(k+1)(k-1)\cdots21$ or $v=us_3^{+}=(k-1)(k+1)k\cdots21$, both $u$ and $v$ lie in $S_{k+1}^{2:1}$. Let $(w,z)$ be an edge in $S_{k+1}^2-\{u,v\}$.
    \medskip

   \textbf{Case 1.} Both~$w$~and~$z$~belong to~$V(S_{k+1}^{2:1})$.

Since $S_{k+1}^{2:1}\cong S_k^2$ and $k\ge4$, the induction hypothesis guarantees a Hamiltonian cycle $C_1$ of $S_{k+1}^{2:1}-\{u,v\}$ containing $(w,z)$. Choose a vertex $w_1=a_1a_2\cdots a_k1$ in $C_1$ with $w_1\notin\{w,z\}$. Let $z_1$ be a neighbor of $w_1$ in $C_1:=\langle w,z,P^1[z,z_1],z_1,w_1,P^1[w_1,w],w \rangle$. It follows that $(w_1,z_1)$ has a coupled pair-edge $(w_1',z_1')$ in $S_k^{2:a_1}$ or $S_k^{2:a_2}$, and so $\langle w_1, w_1', z_1', z_1, w_1 \rangle$ is a $4$-cycle. According to Lemma \ref{11}, the edge $(w_1',z_1')$ is contained in a Hamiltonian cycle $C_2:=\langle w_1',z_1',P^2[z_1',w_1'],w_1'\rangle$ of $S_{n,k}^2$, where $S_{n,k}^2=S_n^{2:a_1}\oplus S_n^{2:a_2}\oplus\dots\oplus S_n^{2:a_k}$. Hence, we obtain a Hamiltonian cycle $C$ of $S_{k+1}^2-\{u,v\}$ containing $(w,z)$, as follows
\[
C:=\langle w_1,P^1[w_1,w],w,z,P^2[z,z_1],z_1,z_1',P^1[z_1',w_1'],w_1',w_1\rangle.
\]

	\textbf{Case 2.} Exactly one of~$w$~and~$z$~belongs to~$V(S_{k+1}^{2:1})$.

Without loss of generality, we may assume $w=a_1a_2\cdots a_k1\in V(S_{k+1}^{2:1})$ and $z\in V(S_{k+1}^{2:j})$,where $j=a_1$ or $a_2$.
Let $w_1=w\circ(1,2)$ and $z_1=z\circ(1,2)$, where $w,w_1\in V(S_{k+1}^{2:1})$, $z,z_1\in V(S_{k+1}^{2:j})$. The induction hypothesis guarantees a Hamiltonian cycle $C_1:=\langle w,w_1,P^1[w_1,w],w\rangle$ of $S_{k+1}^{2:1}-\{u,v\}$. Since the edges $(w,w_1)$ and $(z,z_1)$ form a pair of coupled edges, $\langle w,z,z_1,w_1,w\rangle$ is a $4$-cycle. By Lemma \ref{11}, the edge $(z,z_1)$ is contained in a Hamiltonian cycle $C_2:=\langle z,z_1,P^2[z_1,z],z \rangle$ of $S_{n,k}^2$, where $S_{n,k}^2=S_n^{2:a_1}\oplus S_n^{2:a_2}\oplus\dots\oplus S_n^{2:a_k}$. Hence, we obtain a Hamiltonian cycle $C$ of $S_{k+1}^2-\{u,v\}$ containing $(w,z)$, as follows
\[
C:=\langle w,P^1[w,w_1],w_1,z_1,P^2[z_1,z],z,w\rangle.
\]

	\textbf{Case 3.} Both~$w$~and~$z$~belong to~$V(S_{k+1}^{2:i})$~with~$i\neq 1$.

Since $k+1\ge5$, we may assume $j\in[k+1]\setminus\{1,k,k+1,i\}$. By Lemma \ref{10}, $S_{k+1}^{2:i}$ contains a Hamiltonian cycle $C_1:=\langle w,z,P^1[z,z_j],z_j,w_j,P^1[w_j,w],w\rangle$. Corollary \ref{9} implies that $(w_j,z_j)$ exists a coupled pair-edge $(w_j',z_j')$ in $S_{k+1}^{2:j}$. By Lemma \ref{10} again, $S_{k+1}^{2:j}$ contains a Hamiltonian cycle $C_2:=\langle w_j',z_j',P^2[z_j',z_1^{1}],z_1^{1},w_1^{1},P^2[w_1^{1},w_j'],w_j'\rangle$. Since $u=(k+1)k\cdots21$ has no neighbor in $S_{k+1}^{2:j}$ and $v$ has at most one neighbor in $S_{k+1}^{2:j}$, $w_1^1$ and $z_1^1$ are not adjacent to $u$ or $v$. This implies that $(w_1^{1},z_1^{1})$ has a coupled pair-edge $(u_1,v_1)$ such that $\{u,v\}\cap\{u_1,v_1\}=\varnothing$. By the induction hypothesis we obtain a Hamiltonian cycle $C_3:=\langle u_1,v_1,P^3[v_1,u_1],u_1\rangle$ in $S_{k+1}^{2:i}-\{u,v\}$.

Since $k\ge4$, there exists a vertex $w_3=a_1a_2a_3a_4\cdots a_ki\in V(S_{k+1}^{2:i})$ with $a_1,a_2\in[k+1]\setminus\{1,j\}$ such that $w_3\notin\{w,z\}$. The vertex $w_3$ has a neighbor $z_3$ in $C_1$. Let $P^1[z_3,z_j]$ denote the path in $C_1$ that does not contain $w_3$, and $P^1[z,w_3]$ the path in $C_1$ that does not contain $z_3$. By Lemma \ref{7}, there exists a cycle $C_4:=\langle w_3,P^4[w_3,z_3],z_3,w_3\rangle$ such that $V(C_4)=V(S_{k+1}^{2:t})\cup\{w_3,z_3\}$, where $t\in[k+1]\setminus\{i,j,1\}$. Hence, we obtain a Hamiltonian cycle $C$ of $S_{k+1}^2-\{u,v\}$ containing $(w,z)$, as follows
\[
\begin{aligned}
C:=\langle\; &
v_1,P^3[v_1,u_1],u_1,w_1^{1},P^2[w_1^{1},w_j'],w_j',w_j,P^1[w_j,w],w,z,P^1[z,w_3],w_3,\\
&P^4[w_3,z_3],z_3,P^1[z_3,z_j],z_j,z_j',P^2[z_j',z_1^{1}],z_1^{1},v_1
\rangle,
\end{aligned}
\]see Figure \ref{fig:7}.
       \begin{figure}[h!]
  \centering
  \includegraphics[width=\textwidth]{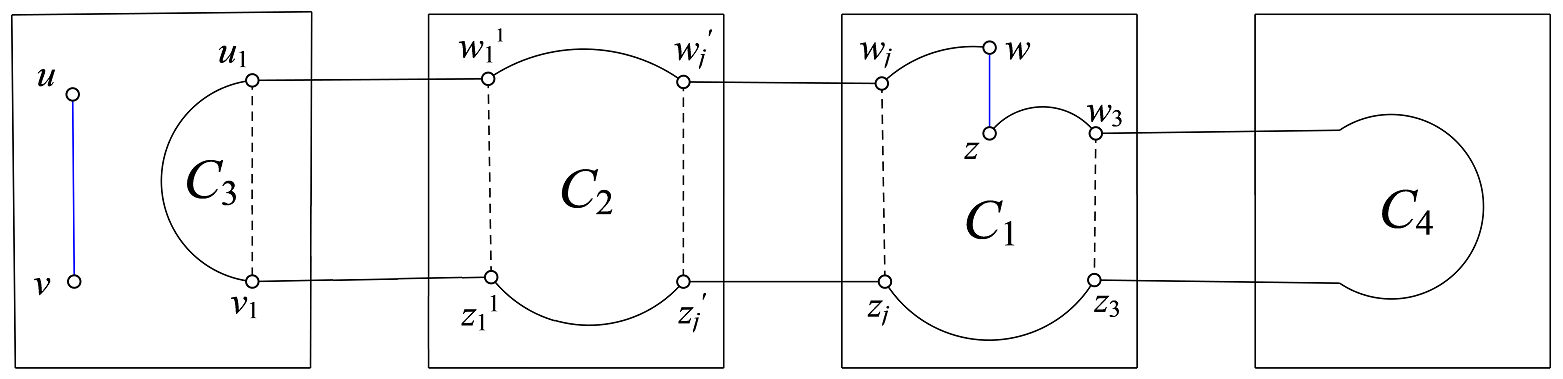}
      \makebox[0.25\textwidth][c]{$S_{n}^{2:1}$}
  \makebox[0.22\textwidth][c]{$S_{n}^{2:j}$}
  \makebox[0.25\textwidth][c]{$S_{n}^{2:i}$}
  \makebox[0.25\textwidth][c]{$S_{n}^2-S_{n}^{2:1}-S_{n}^{2:j}-S_{n}^{2:i}$}
  \caption{Illustration of Lemma \ref{3-2}.}
  \label{fig:7}
  \end{figure}

	\textbf{Case 4.} Both $w$ and $z$ belong to $V(S_{k+1}^{2:i})$ and $V(S_{k+1}^{2:j})$, respectively, where $2\le i,j\le k+1$ and $i\neq j$.

Let $w=a_1a_2\cdots a_ki$ and $z\in v(S_{k+1}^{2:j})$ with $j\in\{a_1,a_2\}$. Set $w_1=w\circ(1,2)\in V(S_{k+1}^{2:i})$ and $z_1=z\circ(1,2)\in V(S_{k+1}^{2:j})$.
Then $\langle w,z,z_1,w_1,w\rangle$ is a $4$-cycle.

Choose $t\in[k+1]\setminus\{1,k,k+1,j\}$. Then $u$ has no neighbor in $S_{k+1}^{2:t}$ and $v$ has at most one neighbor in $S_{k+1}^{2:t}$. By setting $(w,w_1)$ as $(w,z)$ in Case 3, there exists a Hamiltonian cycle
 \[
 \begin{aligned}
C_1:=\langle \; &
w,P^1[w,w_t],w_t,w_t',P^1[w_t',w_1^{1}],w_1^{1},u_1,P^1[u_1,v_1],v_1,\\
&z_1^{1},P^1[z_1^{1},z_t'],z_t',z_t,P^1[z_t,w_1],w_1,w\rangle
\end{aligned}
\] of $(S_{k+1}^{2:i}\oplus S_{k+1}^{2:t}\oplus S_{k+1}^{2:1})-\{u,v\}$.

By Lemma \ref{11}, there is a Hamiltonian cycle $C_2:=\langle z_1,z,P^2[z,z_1],z_1\rangle$ of $S_{k+1}^2-(S_{k+1}^{2:i}\oplus S_{k+1}^{2:t}\oplus S_{k+1}^{2:1})$.
Hence, we obtain a Hamiltonian cycle of $S_{k+1}^2-\{u,v\}$ containing $(w,z)$, as follows
\[
\begin{aligned}
C:= \langle \; &
v_1,P^1[v_1, u_1],u_1,w_1^{1},P^1[w_1^{1}, w_t'],w_t',w_t,P^1[w_t,w],w,z, \\
&P^2[z, z_1],z_1,w_1,P^1[w_1,z_t],z_t,z_t',P^1[z_t',z_1^{1}],z_1^{1},v_1
 \rangle,
\end{aligned}
\]see Figure \ref{fig:8}.
       \begin{figure}[h!]
  \centering
  \includegraphics[width=\textwidth]{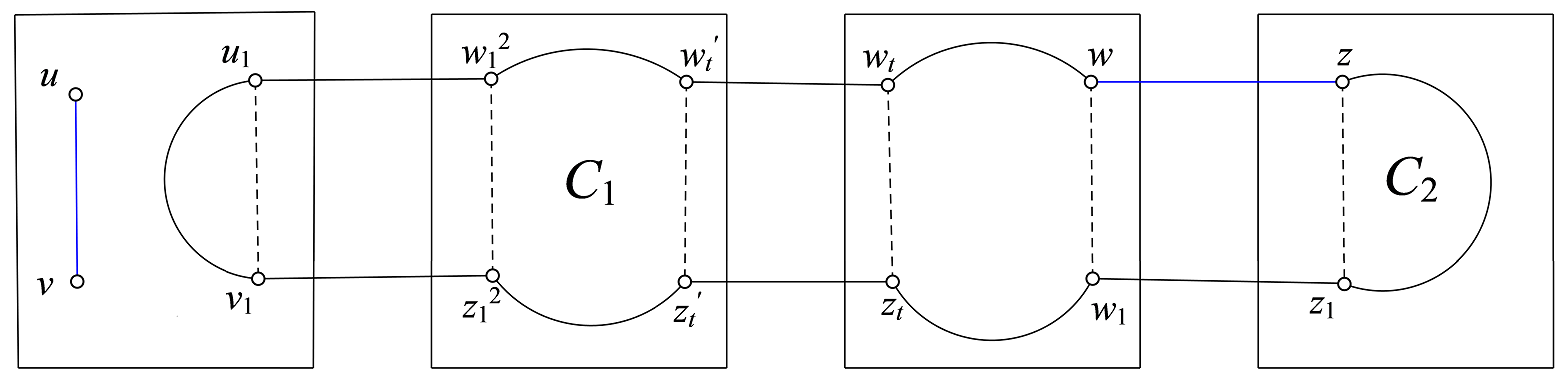}
      \makebox[0.25\textwidth][c]{$S_{n}^{2:1}$}
  \makebox[0.22\textwidth][c]{$S_{n}^{2:t}$}
  \makebox[0.25\textwidth][c]{$S_{n}^{2:i}$}
  \makebox[0.25\textwidth][c]{$S_{n}^2-S_{n}^{2:1}-S_{n}^{2:t}-S_{n}^{2:i}$}
  \caption{Illustration of Lemma \ref{3-2}.}
  \label{fig:8}
  \end{figure}
\end{proof}

	\begin{thm}\label{3-3}
		The~$n$-dimensional split-star network~$S_n^2$~with~$n\ge4$~is 2-DCC vertex~$[3,n!/2]$-pancyclic.
	\end{thm}
\begin{proof}
We prove this theorem by induction on $n$. If $n=4$, then the desired result follows from Lemma \ref{3-1}. Assume that this theorem holds for $n\ge 4$. Pick two distinct vertices $u,v\in V(S_{n+1}^2)$ and an integer $\ell$ satisfying $3\le \ell \le \frac{(n+1)!}{2}$. By induction, it suffices to find two vertex-disjoint cycles $C_1$ and $C_2$ in $S_{n+1}^2$ satisfying
\begin{itemize}
\item $u\in V(C_1)$ and $v\in V(C_2)$;
\item $|C_1|=\ell$ and $|C_2|=(n+1)!-\ell$.
\end{itemize}

Due to the vertex-transitivity of $S_{n+1}^2$, we may assume that $u=(n+1)n\cdots21\in V(S_{n+1}^{2:1})$. The proof proceeds by considering two cases according to whether $v$ lies in $S_{n+1}^{2:1}$ or not. Let $s$ be the integer such that $1\le s\le \left\lceil \frac{n+1}{2}\right\rceil$.

\textbf{Case 1.} $v\in V(S_{n+1}^{2:1})$.

Let $v=y_1y_2\cdots y_n1\in V(S_{n+1}^{2:1})$, where $\{y_1,y_2,\ldots,y_n\}=[n+1]\setminus\{1\}$.

\textbf{Case 1.1.} $3+(s-1)\cdot n!\le \ell \le s\cdot n!-3$.

Let $\ell'=\ell-(s-1)\cdot n!$. Then $3\le \ell' \le n!-3$. By induction, there exist two disjoint cycles $C_1'$ and $C_2'$ in $S_{n+1}^{2:1}$ such that $|C_1'|=\ell'$, $|C_2'|=n!-\ell'$, $u\in V(C_1')$ and $v\in V(C_2')$.

Suppose $s=1$. Then $\ell'=\ell$, and set $C_1=C_1'$. Let $v_1$ be a neighbor of $v$ in $C_2'$. By Lemma \ref{7}, there exists a cycle $C_2''$ in $S_{n+1}^2$ such that $V(C_2'')=V(S_{n+1}^2-S_{n+1}^{2:1})\cup\{v,v_1\}$. Denote paths $P^1[v,v_1]\subseteq C_2'\setminus(v,v_1)$, $P^2[v_1,v]\subseteq C_2''\setminus(v_1,v)$. Set the cycle $C_2:=\langle v,P^1[v,v_1],v_1,P^2[v_1,v],v\rangle$. It follows that $|C_2|=n!-\ell' + n\cdot n! = (n+1)!-\ell$.

Suppose $s\ge2$. Note that $C_1'$ and $C_2'$ cover all vertices of $S_{n+1}^{2:1}$. According to Lemma \ref{12}, there exist two vertices $w=a_1\cdots a_n1$ in $C_1'$ and $z=b_1\cdots b_n1$ in $C_2'$ such that $a_1,a_2,b_1,b_2$ are pairwise distinct. Choose vertices $w_1$ and $z_1$ such that $(w,w_1)\in E(C_1')$ and $(z,z_1)\in E(C_2')$. The edge $(w,w_1)$ has a coupled pair-edge $(w',w_1')$ in $S_{n+1}^{2:m_1}$, where $m_1\in\{a_1,a_2\}$. Similarly, the edge $(z,z_1)$ has a coupled pair-edge $(z',z_1')$ in $S_{n+1}^{2:m_2}$, where $m_2\in\{b_1,b_2\}$. According to Lemma \ref{11}, for integers $i_1,\dots,i_{s-2}\in[n+1]\setminus\{1,m_1,m_2\}$,~$(w',w_1')$ lies on a Hamiltonian cycle $C_1'':=\langle w',w_1',P^1[w_1',w'],w'\rangle$ in $S_{n+1,s-1}^2 = S_{n+1}^{2:m_1}\oplus S_{n+1}^{2:i_1}\oplus \cdots\oplus S_{n+1}^{2:i_{s-2}}$. Then $C_1':=\langle w,w_1,P^1[w_1,w],w\rangle$ can be extended to a cycle \[C_1:=\langle w,w',P^1[w',w_1'],w_1',w_1,P^1[w_1,w],w\rangle\] of length $\ell' + (s-1)\cdot n! = \ell$. Similarly, we can also construct a cycle \[C_2:=\langle z,z',P^2[z',z_1'],z_1',z_1,P^2[z_1,z],z\rangle\] of length $n!-\ell' + (n-s+1)\cdot n! = (n+1)!-\ell$, by concatenating $C_2':=\langle z,z_1,P^2[z_1,z],z\rangle$ with a Hamiltonian cycle $C_2'':=\langle z',z_1',P^2[z_1',z'],z'\rangle$ of $S_{n+1,n-s+1}^2 = S_{n+1}^{2:m_2} \oplus S_{n+1}^{2:j_1} \oplus \cdots \oplus S_{n+1}^{2:j_{n-s}}$, where $\{j_1,\dots,j_{n-s}\}=[n+1]\setminus\{i_1,i_2,\dots,i_2,i_{s-2},1,m_1,m_2\}$. Hence, $C_1$ and $C_2$ constitute a desired $2$-DCC, see Figure \ref{fig:9}.
       \begin{figure}[h!]
  \centering
  \includegraphics[width=\textwidth]{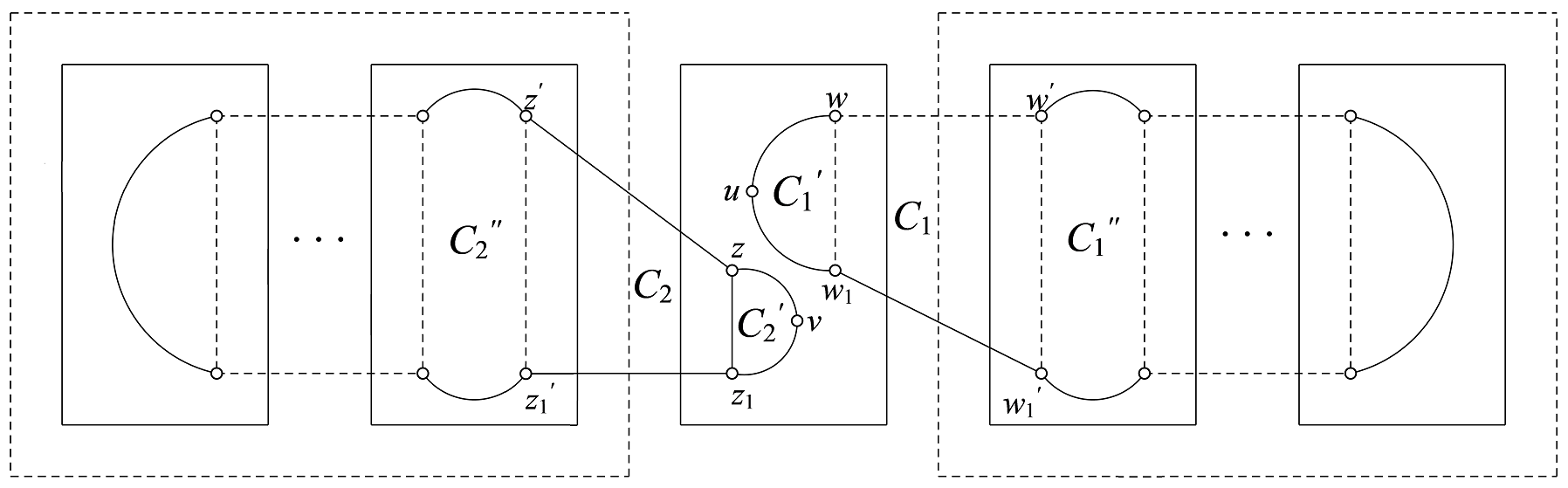}
         \makebox[0.2\textwidth][c]{$S_{n+1,n-s+1}^2$}
  \makebox[0.35\textwidth][c]{$S_{n+1}^{2:1}$}
  \makebox[0.2\textwidth][c]{$S_{n+1,s-1}^2$}
  \caption{Illustration of the proof of Case 1.1 in Theorem \ref{3-3}.}
  \label{fig:9}
  \end{figure}

\textbf{Case 1.2.} $\ell=s\cdot n!-2$.

Let $v_1=v\circ(1,2)=y_2y_1y_3\cdots y_n1$ or $v_1=vs_3^{+}=y_3y_1y_2\cdots y_n1$ such that $v_1\ne u$. Choose $n-s$ integers $\{i_1,i_2,\dots i_{n-s}\}$ in $[n+1]\setminus\{1,y_1\}$ and denote $S_{n+1,n-s+1}^2 = S_{n+1}^{2:y_1}\oplus S_{n+1}^{2:i_1}\oplus \cdots\oplus S_{n+1}^{2:i_{n-s}}$.
The edge $(v,v_1)$ has a coupled pair-edge $(v',v_1')$ in $S_{n+1}^{2:y_1}$. By Lemma \ref{11}, $(v',v_1')$ lies on a Hamiltonian cycle $C_2':=\langle v_1',P^2[v_1',v'],v',v_1'\rangle$ in $S_{n+1,n-s+1}^2$.
The cycle $C_2'$ can be extended to the cycle $C_2:=\langle v,v_1,v_1',P^2[v_1',v'],v',v\rangle$. Thus, $C_2$ is a cycle containing $v$ with length $(n-s+1)\cdot n!+2=(n+1)!-\ell$. Choose two vertices $w=a_1a_2\cdots a_n1$ and $w_1=w\circ(1,2)$, with $a_1\notin\{y_1,i_1,\dots,i_{n-s}\}$. By Lemma~\ref{3-2}, $S_n^{2:1} - \{v, v_1\}$ has a Hamiltonian cycle $C_1':=\langle w,w_1,P^1[w_1,w],w\rangle$ containing $u$.

 If $s=1$, then let $C_1 = C_1'$, which implies $|C_1'| = n! - 2$. Now suppose $s\ge2$. The edge $(w,w_1)$ has a coupled pair-edge $(w',w_1')\in E(S_{n+1}^{2:a_1})$, which lies on a Hamiltonian cycle $C_1'':=\langle w',w_1',P^1[w_1',w'],w'\rangle$ of $S_{n+1,s-1}^2=S_{n+1}^2 - S_{n+1,n-s+1}^2 - S_{n+1}^{2:1}$ by Lemma \ref{11}. This implies that $C_1'$ can be extended to a cycle \[C_1:=\langle w,w',P^1[w',w_1'],w_1',w_1,P^1[w_1,w],w\rangle\] by concatenating it with $C_1''$. Then $C_1$ contains $u$ and $|C_1|=n!-2 + [\,n+1-(n-s+1)\,]\cdot n! = s\cdot n! - 2 = \ell$. Since $C_1$ and $C_2$ are vertex-disjoint, they are a desired $2$-DCC, see Figure \ref{fig:10}.
     \begin{figure}[h!]
  \centering
  \includegraphics[width=\textwidth]{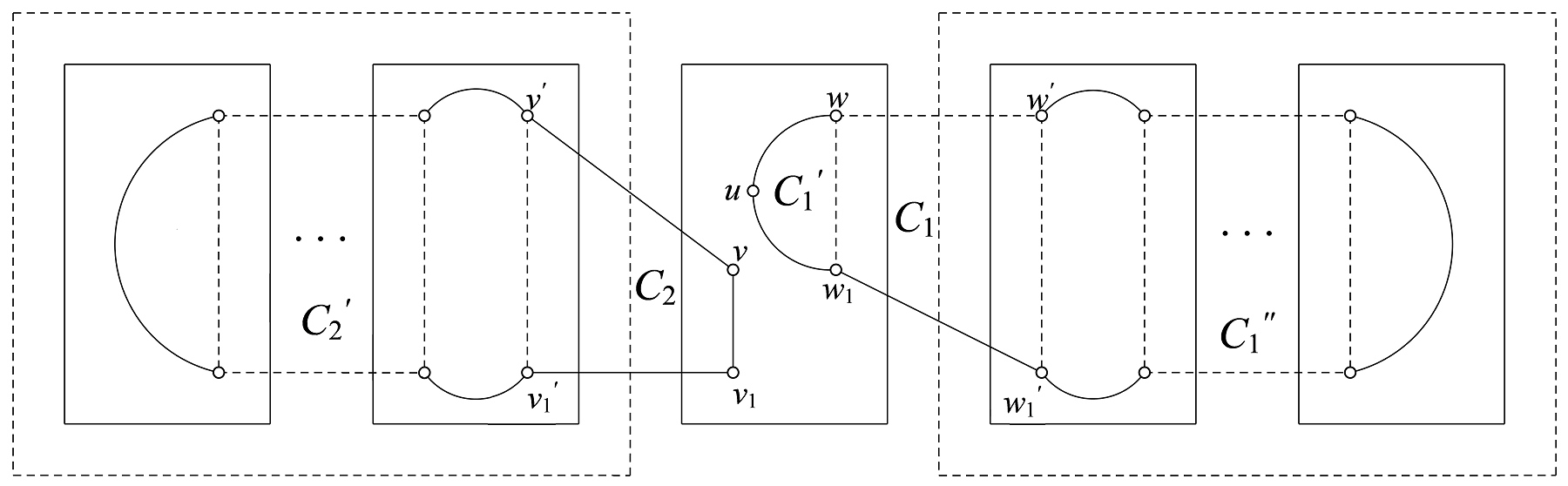}
         \makebox[0.2\textwidth][c]{$S_{n+1,n-s+1}^2$}
  \makebox[0.35\textwidth][c]{$S_{n+1}^{2:1}$}
  \makebox[0.2\textwidth][c]{$S_{n+1,s-1}^2$}
  \caption{Illustration of the proof of Case 1.2 in Theorem \ref{3-3}.}
  \label{fig:10}
  \end{figure}

\textbf{Case 1.3.} $\ell=s\cdot n!-1$.

According to \cite[Lemma~2.8 (i)]{Li25},~$S_{n+1}^{2:1}-\{v\}$ contains a Hamiltonian cycle $C_1'$ and $u\in V(C_1')$. If $s=1$, then let $C_1=C_1'$. By Lemma \ref{6}, there exists a cycle $C_2$ such that $V(C_2)=V(S_{n+1}^{2:y_1}\oplus S_{n+1}^{2:y_2}\oplus S_{n+1}^{2:i_1}\oplus \cdots \oplus S_{n+1}^{2:i_{n-2}})\cup\{v\}$ for integers $i_1,\ldots,i_{n-2}\in[n+1]\setminus\{1,y_1,y_2\}$. Then $|C_2|=n\cdot n!+1=(n+1)!-(n!-1)=(n+1)!-\ell$.
Hence, $C_1$ and $C_2$ form the desired $2$-DCC.

Suppose $s\ge2$. Choose a vertex $w=a_1a_2\cdots a_n1$ in $C_1'$ such that $\{a_1,a_2\}\cap\{y_1,y_2\}=\emptyset$, and $w$ has a neighbor $w_1$ in $C_1'$. Then $(w,w_1)$ has a coupled pair-edge $(w',w_1')$ in $S_{n+1}^{2:m_1}$, where $m_1=a_1$ or $a_2$. Choose $n-s-1$ integers $i_1,i_2,\dots,i_{n-s-1}\in[n+1]\setminus\{1,y_1,y_2,m_1\}$.
According to Lemma \ref{6}, there exists a cycle $C_2$ in $S_{n+1}^2$ containing $v$ such that $V(C_2)=V(S_{n+1}^{2:y_1}\oplus S_{n+1}^{2:y_2}\oplus S_{n+1}^{2;i_1}\oplus\cdots\oplus S_{n+1}^{2;i_{n-s-1}})\cup\{v\}$. Then $|C_2|=(n-s+1)\cdot n!+1$. By Lemma \ref{11},~$(w',w_1')$ lies on a Hamiltonian cycle $C_1'':=\langle w',w_1',P^1[w_1',w'],w'\rangle$ of $S_{n+1,s-1}^2=S_{n+1}^{2:m_1}\oplus S_{n+1}^{2:j_1}\oplus S_{n+1}^{2:j_2}\oplus\cdots\oplus S_{n+1}^{2:j_{s-2}}$, where $\{j_1,j_2,\dots,j_{s-2}\}\in[n+1]\setminus\{i_1,i_2,\dots,i_{n-s-1},1,y_1,y_2,m_1\}$. Hence $C_1':=\langle w,w_1,P^1[w_1,w],w\rangle$ can be extended to a cycle \[C_1:=\langle w,w',P^1[w',w_1'],w_1',w_1,P^1[w_1,w],w\rangle\] by concatenating with $C_1''$. Consequently, $|C_1|=(s-1)\cdot n!+n!-1=s\cdot n!-1=(n+1)!-|C_2|$. Clearly, $C_1$ and $C_2$ are vertex-disjoint, see Figure \ref{fig:11}.
     \begin{figure}[h!]
  \centering
  \includegraphics[width=\textwidth]{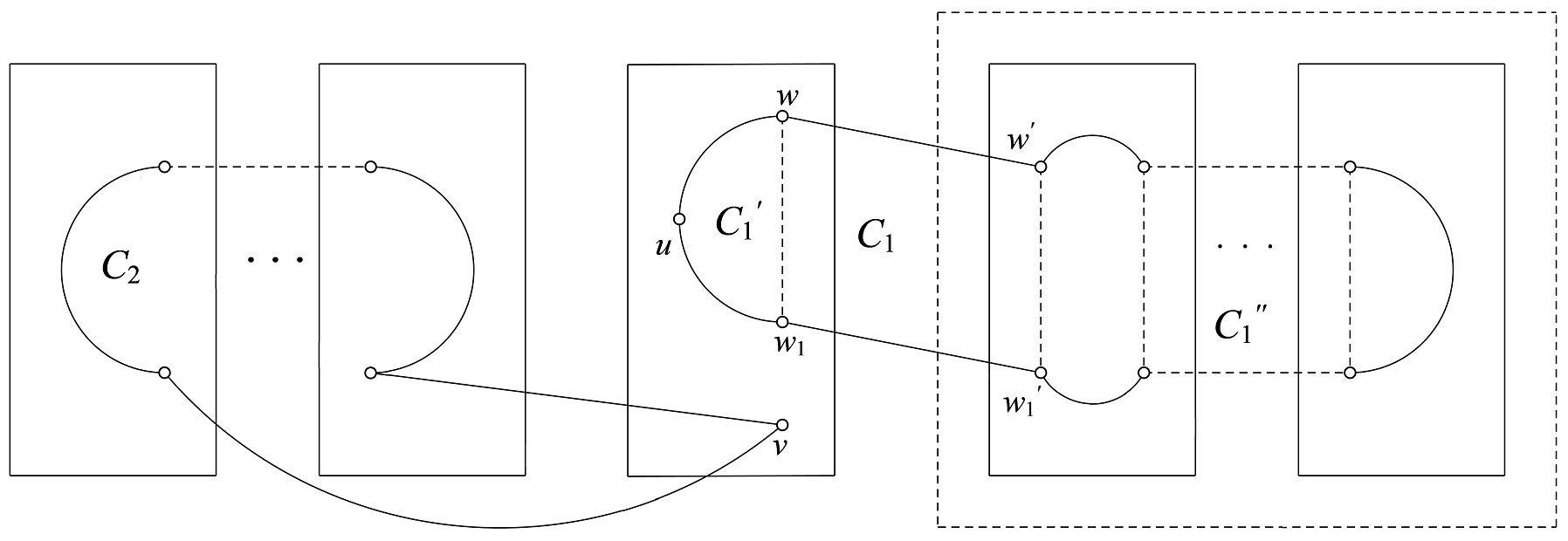}
      \makebox[0.19\textwidth][c]{$S_{n+1}^{2:y_2}$}
  \makebox[0.19\textwidth][c]{$S_{n+1}^{2:y_1}$}
  \makebox[0.19\textwidth][c]{$S_{n+1}^{2:1}$}
  \makebox[0.40\textwidth][c]{$S_{n+1,s-1}^{2}$}
  \caption{Illustration of the proof of Case 1.3 in Theorem \ref{3-3}.}
  \label{fig:11}
  \end{figure}

\textbf{Case 1.4.} $\ell=s\cdot n!$.

Let $v_1=v\circ(1,2)=y_2y_1y_3\cdots y_n1$ or $v_1=vs_3^{+}=y_3y_1y_2\cdots y_n1$ such that $v_1\ne u$. Choose two vertices $w=vs_4^{+}=y_4y_1y_3y_2\cdots1$ and $w_1=w\circ(1,2)=y_1y_4y_3y_2\cdots1$.
The edges $(v,v_1)$ and $(w,w_1)$ belong to $S_{n+1}^{2:1}$, and $\{w,w_1\}\cap\{v,v_1\}=\emptyset$.
Applying Lemma~\ref{3-2}, we obtain a Hamiltonian cycle $C_1':=\langle w,P^1[w,w_1],w_1,w \rangle$ in $S_{n+1}^{2:1} - \{v, v_1\}$.
Since $u\notin\{v,v_1\}$, it follows that $u\in V(C_1')$.
Moreover, $(v', v_1')$ and $(w', w_1')$ are the coupled pair-edges of
$(v, v_1)$ and $(w, w_1)$ in $S_{n+1}^{2:y_1}$, respectively,
and $\{w', w_1'\} \cap \{v', v_1'\} = \emptyset$.
Applying Lemma~\ref{3-2} again, we obtain a Hamiltonian cycle
$C_2'$ in $S_{n+1}^{2:y_1}-\{w',w_1'\}$ containing the edge $(v',v_1')$.

Suppose $s=1$. It implies $\ell=n!$.
Let
\[
C_1:=\langle w',w,P^1[w,w_1],w_1,w_1',w'\rangle.
\]
Thus $|C_1|=\ell$ and $u\in V(C_1)$.
Choose a vertex $v_2=y_2y_3\cdots 1y_1$ in $C_2'$. Let $v_3$ be a neighbor of $v_2$ in $C_2':=\langle v',v_1',P^2[v_1',v_2],v_2,v_3,P^2[v_3,v'],v'\rangle$.
By Lemma~\ref{7}, $(v_2,v_3)$ lies on a cycle $C_2'':=\langle v_2,v_3,P^2[v_3,v_2],v_2\rangle$
such that $V(C_2'')=V(S_{n+1}^2-S_{n+1}^{2:1}-S_{n+1}^{2:y_1})\cup\{v_2,v_3\}$.
Hence, let
\[
C_2:=\langle v,v_1,v_1',P^2[v_1',v_2],v_2,
P^2[v_2,v_3],v_3,P^2[v_3,v'],v',v\rangle.
\]
Then $C_2$ contains $v$ and $|C_2|=n\cdot n!=(n+1)!-\ell$.
Therefore, $C_1$ and $C_2$ form the desired $2$-DCC.

Suppose $s\ge2$. Since $|C_1'|=n!-2$, there exists a vertex $w_2=y_2y_3\cdots y_ny_11$ or $y_3y_2\cdots y_ny_11$ in $C_1'$ such that
$w_2\neq u$. The vertex $w_2$ has a neighbor $w_3$ in $C_1':=\langle w,w_1,P^1[w_1,w_2],w_2,w_3,P^1[w_3,w],w\rangle$.
The edge $(w_2,w_3)$ has a coupled pair-edge $(w_2',w_3')$ in
$S_{n+1}^{2:m_1}$, where $m_1=y_2$ or $y_3$. By Lemma~\ref{11},
$(w_2',w_3')$ lies on a Hamiltonian cycle $C_1'':=\langle w_2',w_3',P^1[w_3',w_2'],w_2'\rangle$ of
$S_{n+1,s-1}^{2}=S_{n+1}^{2:m_1}\oplus S_{n+1}^{2:i_1}\oplus\cdots\oplus S_{n+1}^{2:i_{s-2}}$,
where $i_1,\dots,i_{s-2}$ are integers in $[n+1]\setminus\{1,y_1,m_1\}$.
Let
\[
C_1:=\langle w',w_1',w_1,P^1[w_1,w_2],w_2,w_2',P^1[w_2',w_3'],w_3',w_3,P^1[w_3,w],w,w'\rangle.
\]
Then $C_1$ contains $u$ and has length $s\cdot n!$.

In $C_2'$ there exists a vertex $z=y_4y_i\cdots1y_1$, where $y_i\in\{y_2,y_3\}\setminus\{m_1\}$.
The vertex $z$ has a neighbor $z_1$ in $C_2':=\langle v',v_1',P^2[v_1',z_1],z_1,z,P^2[z,v'],v'\rangle$. The edge $(z,z_1)$
has a coupled pair-edge $(z',z_1')$ in $S_{n+1}^{2:m_2}$, where
$m_2=y_4$ or $y_i$. By Lemma~\ref{11}, $(z',z_1')$ lies on a Hamiltonian
cycle $C_2'':=\langle z',z_1',P^2[z_1',z'],z'\rangle$ of $S_{n+1,n-s}^{2}=S_{n+1}^{2}- S_{n+1,s-1}^{2}-S_{n+1}^{2:1}-S_{n+1}^{2:y_1}$.
Let
\[
C_2:=\langle v,v_1,v_1',P^2[v_1',z_1],z_1,z_1',P[z_1',z'],z',z,P^2[z,v'],v',v\rangle.
\]
Then $C_2$ contains $v$ and has length $(n-s+1)\cdot n!=(n+1)!-\ell$, as illustrated in Figure \ref{fig:12}. Hence, $C_1$ and $C_2$
form the desired $2$-DCC.
     \begin{figure}[h!]
  \centering
  \includegraphics[width=\textwidth]{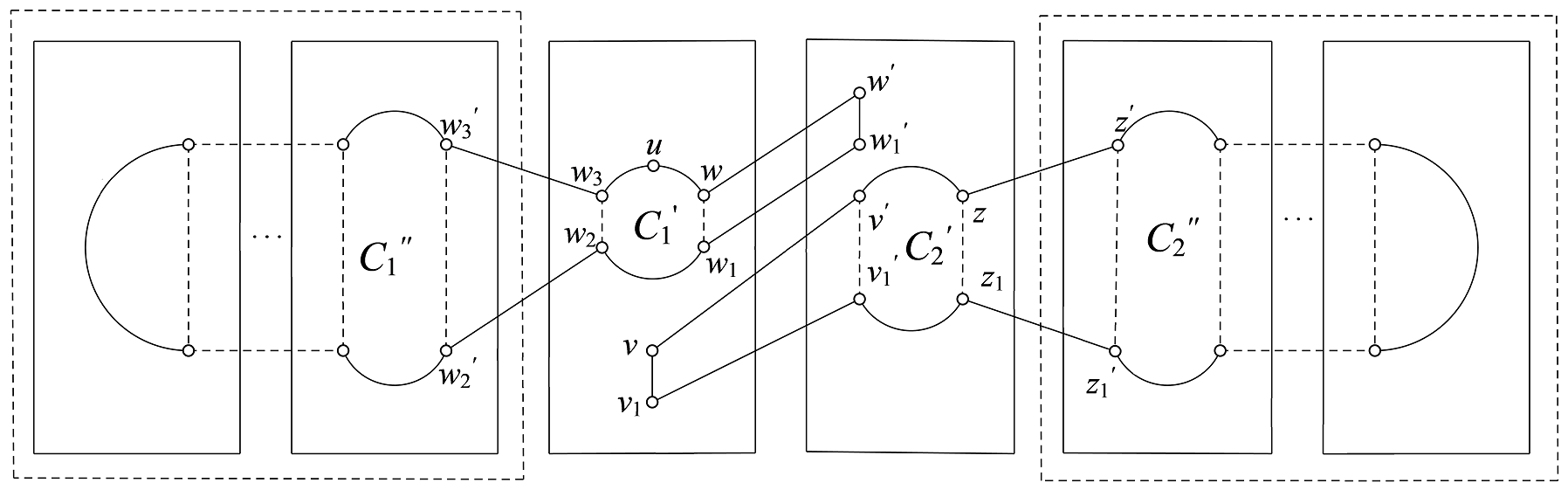}
    \makebox[0.23\textwidth][c]{$S_{n+1,s-1}^2$}
  \makebox[0.15\textwidth][c]{$S_{n+1}^{2:1}$}
  \makebox[0.15\textwidth][c]{$S_{n+1}^{2:y_1}$}
  \makebox[0.23\textwidth][c]{$S_{n+1,n-s}^2$}
  \caption{Illustration of the proof of Case 1.4 in Theorem \ref{3-3}.}
  \label{fig:12}
  \end{figure}

\textbf{Case 1.5.} $\ell=s\cdot n!+1$.

Suppose $s = 1$, then $\ell = n! + 1$. By~\cite[Lemma 2.8 (i)]{Li25}, $S_{n+1}^{2:1} - \{v\}$ admits a Hamiltonian cycle $C_1'$ containing $u$. Let $w = y_3y_4\cdots y_1y_21\in V(C_1')$, and let $w_1$ be a neighbor of $w$ on $C_1':=\langle w,w_1,P^1[w_1,w],w\rangle$. The edge $(w, w_1)$ has a coupled pair-edge $(w', w_1')$ in $S_{n+1}^{2:m}$, where $m = y_3$ or $y_4$. Let
\[
C_1 := \langle w', w_1', w_1, P^1[w_1, w], w, w' \rangle.
\]
Then $|C_1|=n!+1$.

 By \cite[Lemma 2.8 (ii) and (iii)]{Li25}, $S_{n+1}^{2:m} - \{w', w_1'\}$ contains a Hamiltonian cycle $C_2' := \langle z, z_1, P^2[z_1,z], z \rangle$, where $z = y_1 y_j \cdots 1 m$ with $j \in \{3, 4\} \setminus \{m\}$, and $z_1$ is a neighbor of $z$ in $C_2'$.

Regardless of whether $(z, z_1)$ is a $(1,2)$-edge or an $i$-edge, the vertex $z_1$ has a neighbor in either $S_{n+1}^{2:y_j}$ or $S_{n+1}^{2:y_1}$, while $z$ has a neighbor in the other subgraph. Since the proofs are similar for $z_1=z\circ(1,2),zs_{i}^+$ and $zs_{i}^-$, we may assume $z_1=z\circ(1,2)=y_jy_1\cdots1m$. Then we let
\[
\begin{aligned}
u_1 &= z_1s_{n+1}^- = y_j m \cdots 1 y_1, &\quad
u_j &= zs_{n+1}^+ =  m y_1\cdots 1 y_j,\\
v_1 &= v s_{n+1}^- = y_2 1 \cdots y_n y_1, &\quad
v_2 &= v s_{n+1}^+ = 1 y_1 \cdots y_n y_2.
\end{aligned}
\]

When $n=4$, let $u_2 = y_j \cdots y_2$ and $v_j = u_2s_{n+1}^-$. Denote $P[v_j,u_2]=(v_j,u_2)$.
When $n\ge 5$, let $u_2=y_5 \cdots y_2$ and $v_5 = u_2s_{n+1}^-$. For each $i \in [5, n-1]$, let $u_i = y_{i+1} \cdots y_2$ and $v_{i+1} = u_is_{n+1}^-$. Let $u_n = y_j \cdots y_n$ and $v_j =u_ns_{n+1}^-$. Denote $P[v_j,u_2]=\langle v_j,u_n,P[u_n,v_n],v_n,\ldots,u_5,P[u_5,v_5],v_5,u_2\rangle$.
Whether $n=4$ or $n\ge 5$, set $P[u_i,v_i]$ be a Hamiltonian path in $S_{n+1}^{2:y_i}$. Let
\[
\begin{aligned}
C_2:= \langle \; & v,v_1,P[v_1,u_1],u_1,z_1,P^2[z_1,z],z,u_j,P[u_j,v_j],v_j,P[v_j,u_2],u_2,P[u_2,v_2],v_2,v \rangle.
\end{aligned}
\] Hence, $C_1$ and $C_2$ form the desired $2$-DCC, as illustrated in Figure \ref{fig:13}.
     \begin{figure}[h!]
  \centering
  \includegraphics[width=\textwidth]{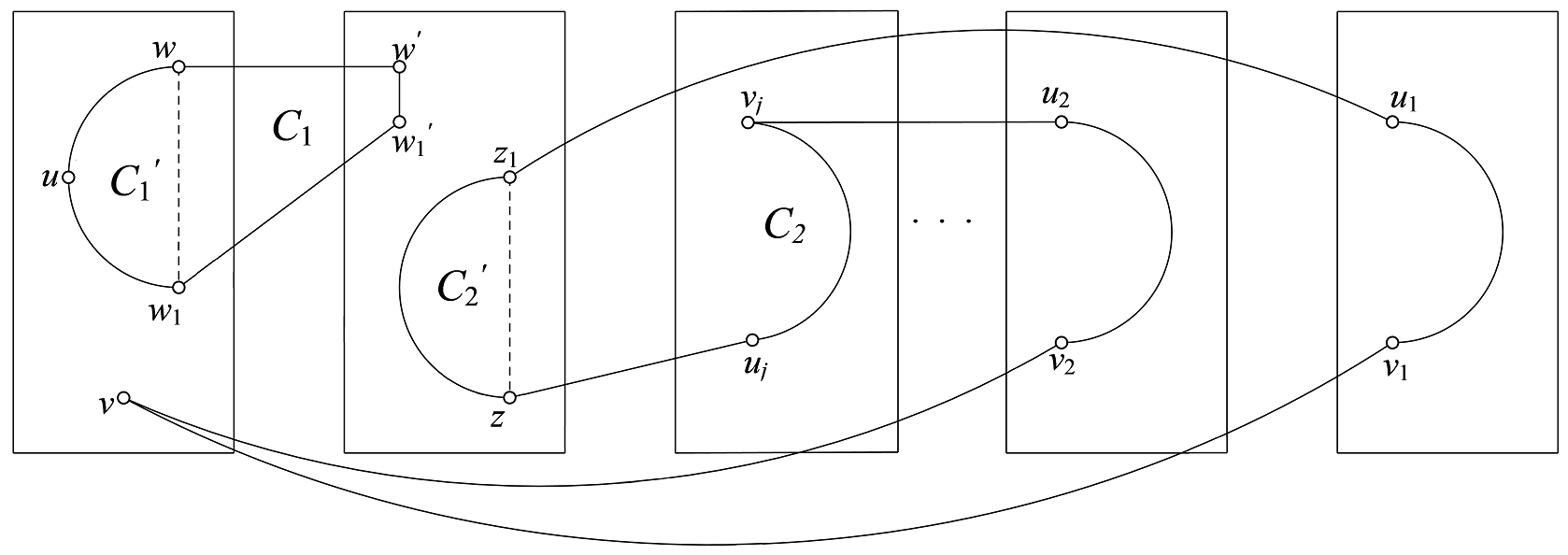}
        \makebox[0.18\textwidth][c]{$S_{n+1}^{2:1}$}
  \makebox[0.18\textwidth][c]{$S_{n+1}^{2:m}$}
  \makebox[0.18\textwidth][c]{$S_{n+1}^{2:y_j}$}
  \makebox[0.18\textwidth][c]{$S_{n+1}^{2:y_2}$}
  \makebox[0.18\textwidth][c]{$S_{n+1}^{2:y_1}$}
  \caption{Illustration of the proof of Case 1.5 in Theorem \ref{3-3}.}
  \label{fig:13}
  \end{figure}

For $s \ge 2$, by~\cite[Lemma 2.8 (i)]{Li25}, $S_{n+1}^{2:1} - \{u\}$ admits a Hamiltonian cycle $C_2'$ containing $v$. Since $u = (n+1)n \cdots 2 1$, according to Lemma~\ref{6}, there exists a cycle $C_1$ in $S_{n+1}^{2:1}$ such that $V(C_1)=V(S_{n+1}^{2:n+1}\oplus S_{n+1}^{2:n}\oplus S_{n+1}^{2:i_1}\oplus S_{n+1}^{2:i_2}\oplus\cdots\oplus S_{n+1}^{2:i_{s-2}})\cup \{u\}$,
where $\{i_1, \dots, i_{s-2}\} \subseteq [n+1] \setminus \{1, n+1, n\}$ and $|C_1|=s\cdot n!+1$.
Since $n \ge 4$, $C_2'$ contains a vertex $w=a_1a_2\cdots a_n1$ such that $\{a_1, a_2\} \subseteq [n+1] \setminus \{1,n+1,n,i_1,\dots, i_{s-2}\}$. Let $w_1$ be a neighbor of $w$ in $C_2':=\langle w,w_1,P^2[w_1,w],w\rangle$. The edge $(w, w_1)$ has a coupled pair-edge $(w', w_1')$ in $S_{n+1}^{2:m_1}$, where $m_1 = a_1$ or $a_2$. By Lemma~\ref{11}, $(w', w_1')$ lies on a Hamiltonian cycle $C_2'':=\langle w',w_1',P^2[w_1',w'],w'\rangle$ of $S_{n+1,n-s}^{2} = S_{n+1}^{2:a_1} \oplus S_{n+1}^{2:a_2} \oplus S_{n+1}^{2:j_1} \oplus \cdots \oplus S_{n+1}^{2:j_{n-s-2}}$ with $\{j_1, \dots, j_{n-s-2}\} \subseteq [n+1] \setminus \{1, n+1, n, i_1, \dots, i_{s-2}, a_1, a_2\}$.
Finally, let
\[
C_2 := \langle w, P^2[w, w_1], w_1, w_1', P^2[w_1', w'], w', w \rangle.
\]
Then $C_2$ contains $v$ and has length $(n-s) \cdot n! + n! - 1 = (n+1)! - \ell$.
Hence, $C_1$ and $C_2$ form the desired $2$-DCC, as illustrated in Figure \ref{fig:14}.
     \begin{figure}[h!]
  \centering
  \includegraphics[width=\textwidth]{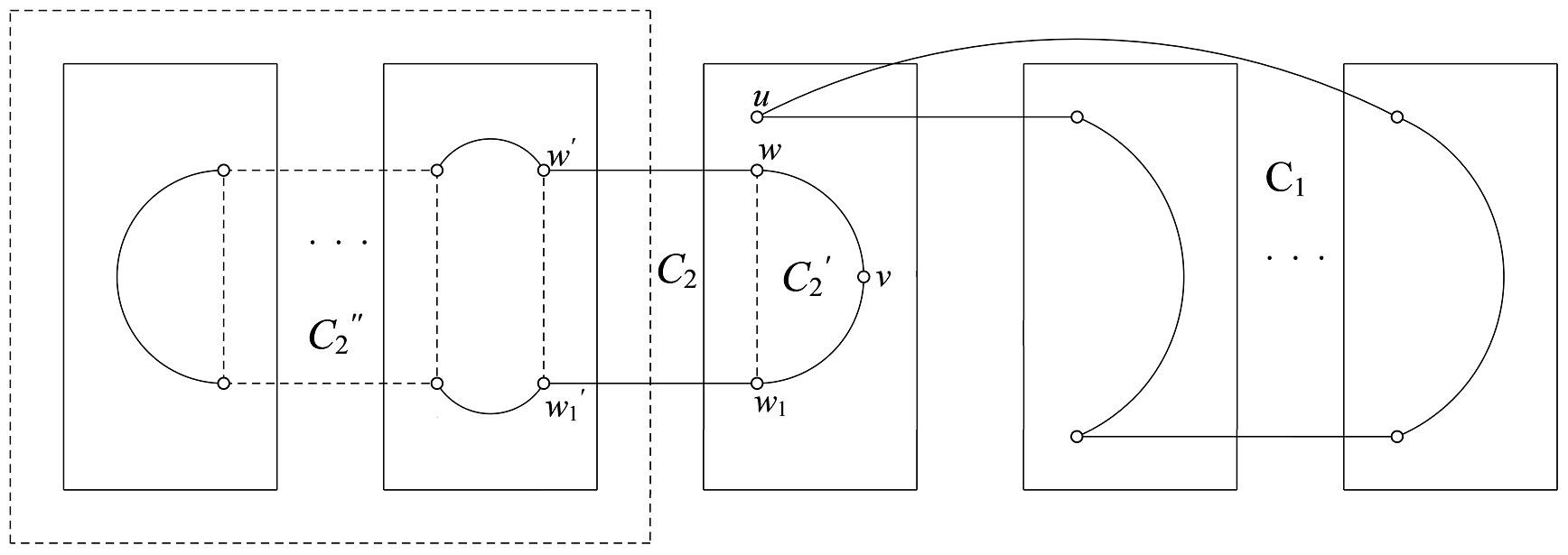}
    \makebox[0.4\textwidth][c]{$S_{n+1,n-s}^{2}$}
      \makebox[0.19\textwidth][c]{$S_{n+1}^{2:1}$}
  \makebox[0.2\textwidth][c]{$S_{n+1}^{2:n}$}
  \makebox[0.18\textwidth][c]{$S_{n+1}^{2:{n+1}}$}

  \caption{Illustration of the proof of Case 1.5 in Theorem \ref{3-3}.}
  \label{fig:14}
  \end{figure}

\textbf{Case 1.6.} $\ell=s\cdot n!+2$.

From the two vertices $u\circ(1,2)=n(n+1)(n-1)\cdots21$ and $us_3^{+}=(n-1)(n+1)n(n-2)\cdots21$, we choose one that is different from \(v\) and denote it by \(u_1\). The edge $(u,u_1)$ has a coupled pair-edge $(u',u_1')$ in $S_{n+1}^{2:n+1}$.
According to Lemma~\ref{11}, $(u',u_1')$ lies on a Hamiltonian cycle $C_1':=\langle u',u_1',P^1[u_1',u'],u'\rangle$ of
$S_{n+1,s}^2 =S_{n+1}^{2:n+1}\oplus S_{n+1}^{2:i_1}\oplus\cdots\oplus S_{n+1}^{2:i_{s-1}}$,
where $i_1,i_2,\dots,i_{s-1}\in[n+1]\setminus\{1,n+1\}$.
Let
\[
C_1:=\langle u,u_1,u_1',P^1[u_1',u'],u',u\rangle .
\] Then $C_1$ contains $u$ and has length $s\cdot n!+2$.

By Lemma~\ref{3-2}, $S_{n+1}^{2:1}-\{u,u_1\}$ exists a Hamiltonian cycle $C_2':=\langle z,P^2[z,z_1],z_1,z\rangle$ of length $n!-2$, where
$z=a_1a_2\cdots a_n1$ with $a_1,a_2\in[n+1]\setminus\{1,n+1,i_1,\dots,i_{s-1}\}$ and $z_1=z\circ(1,2)$.
The edge $(z,z_1)$ has a coupled pair-edge $(z',z_1')$ in $S_{n+1}^{2:a_1}$, which lies on a Hamiltonian cycle $C_2'':=\langle z',z_1',P^2[z_1',z'],z'\rangle$ of
$S_{n+1,n-s}^2=S_{n+1}^{2:a_1}\oplus S_{n+1}^{2:a_2}\oplus S_{n+1}^{2:j_1}\oplus\cdots\oplus S_{n+1}^{2:j_{n-s-2}}$, where $j_1,j_2,\ldots,j_{n-s-2}\in[n+1]\setminus\{1,n+1,a_1,a_2,i_1,\dots,i_{s-1}\}$.
Let
\[
C_2:=\langle z,P^2[z,z_1],z_1,z_1',P^2[z_1',z'],z',z\rangle .
\]
Thus $C_2$ contains $v$ and has length $(n-s)\cdot n!+n!-2=(n+1)!-\ell$.
Hence, $C_1$ and $C_2$ form the desired $2$-DCC, as illustrated in Figure \ref{fig:15}.
     \begin{figure}[h!]
  \centering
  \includegraphics[width=\textwidth]{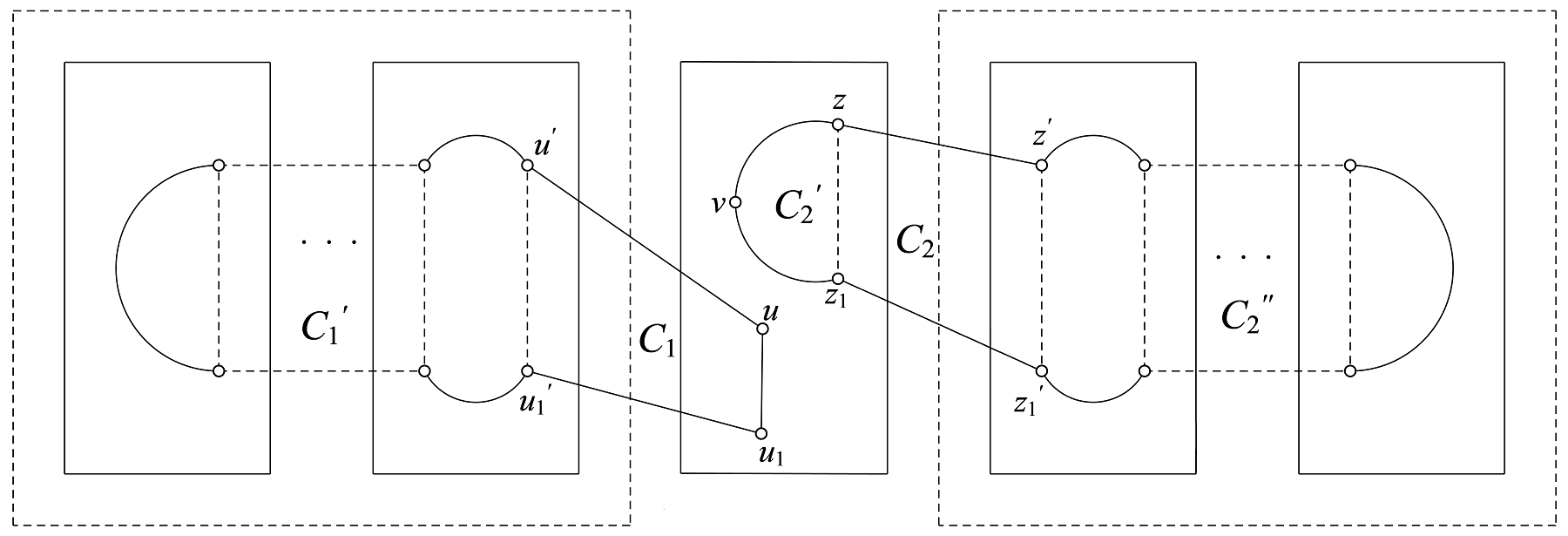}
       \makebox[0.2\textwidth][c]{$S_{n+1,s}^2$}
  \makebox[0.35\textwidth][c]{$S_{n+1}^{2:1}$}
  \makebox[0.2\textwidth][c]{$S_{n+1,n-s}^2$}
  \caption{Illustration of the proof of Case 1.6 in Theorem \ref{3-3}.}
  \label{fig:15}
  \end{figure}

\textbf{Case 2.} $v\in V(S_{n+1}^{2:i})$, where $i\neq 1$.

\textbf{Case 2.1.} $3+(s-1)\cdot n!\le \ell \le s\cdot n!-3$.

Suppose $s=1$. Let $u_1=n(n+1)\cdots21\in V(S_{n+1}^{2:1})$.
By the induction hypothesis, there exist two vertex-disjoint cycles $C_1'$ and $C_2'$ in $S_{n+1}^{2:1}$ such that
$|C_1'|=\ell$ and $|C_2'|=n!-\ell$, where $u\in V(C_1')$ and $u_1\in V(C_2')$.
Let $u_1'$ be a neighbor of $u_1$ in $C_2':=\langle u_1,u_1',P^2[u_1',u_1],u_1\rangle$.
By Lemma~\ref{7}, there exists a cycle $C_2'':=\langle u_1,u_1',P^2[u_1',u_1],u_1\rangle$ in $S_{n+1}^2$ such that
$V(C_2'') = V(S_{n+1}^2 - S_{n+1}^{2:1})\cup\{u_1,u_1'\}$.
Let $C_2:=\langle u_1',P^2[u_1',u_1],u_1,P^2[u_1,u_1'],u_1'\rangle$. Then $C_2$ contains $v$ and $|C_2|=n\cdot n!+n!-\ell=(n+1)!-\ell$.
Hence, $C_1$ and $C_2$ form the desired $2$-DCC.

Suppose $s\ge 2$. Let $\ell'=\ell-(s-1)\cdot n!$. It follows that $3\le\ell'\le n!-3$. We choose two vertices $w=1a_1\cdots a_{n-1}j$ and
$z=ib_1\cdots b_{n-1}j$ in $S_{n+1}^{2:j}$, where $b_1\notin\{1,a_1\}$.
By the induction hypothesis, $S_{n+1}^{2:j}$ has two disjoint cycles $C_1'$ and $C_2'$ such that
$|C_1'|=\ell'$ and $|C_2'|=n!-\ell'$, where $w\in V(C_1')$ and $z\in V(C_2')$.
Let $w_1$ be a neighbor of $w$ in $C_1':=\langle w,w_1,P^1[w_1,w],w\rangle$.

If $w_1\in\{w\circ(1,2),ws_k^{+}\}$ for $k\in[3,n]$, then $(w,w_1)$ has a coupled pair-edge $(w',w_1')$ in $S_{n+1}^{2:1}$. By Lemma~\ref{11}, for integers $i_1,\dots,i_{s-2}\in[n+1]\setminus\{1,i,j,b_1\}$, the edge $(w',w_1')$ lies on a Hamiltonian cycle $C_1'':=\langle w',w_1',P^1[w_1',w'],w'\rangle$ of $S_{n+1,s-1}^2 = S_{n+1}^{2:1} \oplus S_{n+1}^{2:i_1}\oplus\cdots\oplus S_{n+1}^{2:i_{s-2}}$. Thus $C_1'$ can be extended to a cycle \[C_1:=\langle w,w',P^1[w',w_1'],w_1',w_1,P^1[w_1,w],w\rangle\] by concatenating it with $C_1''$. Thus $C_1$ contains $u$ and has length $\ell'+(s-1)\cdot n!=\ell$. We choose a vertex $z'$ in $C_2':=\langle z,z',P^2[z',z],z\rangle$. By Lemma \ref{7}, there exists a cycle $C_2'':=\langle z',z,P^2[z,z'],z'\rangle$ in $S_{n+1}^2$ such that $V(C_2'')=V(S_{n+1}^2-S_{n+1,s-1}^2-S_{n+1}^{2:j})\cup\{z,z'\}$. Let $C_2:=\langle z',P^2[z',z],z,P^2[z,z'],z'\rangle$. Thus $C_2$ contains $v$ and has length $n!-\ell' + (n-s+1)\cdot n! = (n+1)!-\ell$. Hence, $C_1$ and $C_2$ form the desired $2$-DCC, as illustrated in Figure \ref{fig:16}.
     \begin{figure}[h!]
  \centering
  \includegraphics[width=\textwidth]{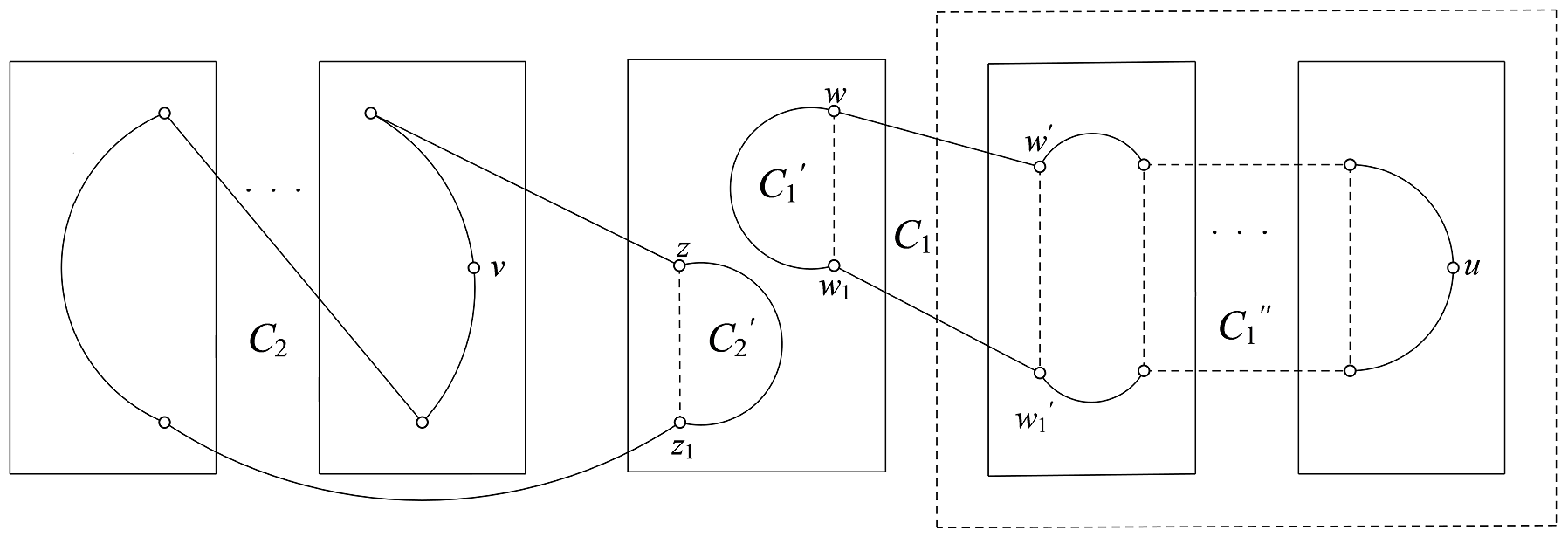}
    \makebox[0.18\textwidth][c]{$S_{n+1}^{2:b_1}$}
  \makebox[0.18\textwidth][c]{$S_{n+1}^{2:i}$}
  \makebox[0.2\textwidth][c]{$S_{n+1}^{2:1}$}
  \makebox[0.4\textwidth][c]{$S_{n+1,s-1}^{2}$}
  \caption{Illustration of the proof of Case 2.1 in Theorem \ref{3-3}.}
  \label{fig:16}
  \end{figure}

If $w_1=ws_k^{-}$ for $k\in[3,n]$, then $(w,w_1)$ has a coupled pair-edge $(w',w_1')$ in $S_{n+1}^{2:a_1}$. Define a mapping $\sigma=(1~a_1)$ that swaps $1$ and $a_1$ in all the vertices of $S_{n+1}^{2:j}$. It follows that $\sigma=(1~a_1)$ preserves all adjacencies of $S_{n+1}^{2:j}$. Hence, there exist two vertex-disjoint cycles $C_1''$ and $C_2''$ in $S_{n+1}^{2:j}$ such that $C_1'' \cong C_1'$ and $C_2'' \cong C_2'$. Let $\sigma(w)=w_2$, $\sigma(w_1)=w_3$, and $\sigma(z)=z_1$. Then $(w_2,w_3)\in E(C_1'')$, $z_1\in V(C_2'')$. Furthermore, $(w_2,w_3)$ has a coupled pair-edge $(w_2',w_3')$ in $S_{n+1}^{2:1}$. Since the proof is similar to the case where $w_1\in\{w\circ(1,2),ws_k^{+}\}$, we also obtain two cycles $C_1$ and $C_2$ such that $C_1$ and $C_2$ form the desired $2$-DCC.

\textbf{Case 2.2.} $\ell=s\cdot n!-2$.

Choose two adjacent vertices $w=ia_1\cdots a_{n-1}1$ and $z=w\circ(1,2)$ in $S_{n+1}^{2:1}$ such that $u\notin\{w,z\}$. Then $(w,z)$ has a coupled pair-edge $(w',z')$ in $S_{n+1}^{2:i}$, which lies on a Hamiltonian cycle $C_2':=\langle w',z',P^2[z',w'],w'\rangle$ of $S_{n+1,n-s+1}^{2}=S_{n+1}^{2:i} \oplus S_{n+1}^{2:i_1}\oplus\cdots\oplus S_{n+1}^{2:i_{n-s}}$, where $\{i_1,\ldots,i_{n-s}\}\subseteq [n+1]\setminus\{1,i\}$. Let $C_2:=\langle z,z',P^2[z',w'],w',w,z\rangle$. Then $|C_2|=(n+1-s)\cdot n!+2=(n+1)!-\ell$ and $v\in V(C_2)$.

Suppose $s=1$. In view of \cite[Lemma 2.8 (ii) and (iii)]{Li25},~$S_{n+1}^{2:1}-\{w,z\}$ has a hamiltonian cycle $C_1$ containing $u$ and $|C_1|=n!-2$.
Suppose $s\ge 2$. Let $w_1=b_1b_2\cdots b_n1$ and $z_1=w_1\circ(1,2)$, where $b_1\in[n+1]\setminus\{i,i_1,\ldots,i_{n-s}\}$. By Lemma~\ref{3-2}, there exists a Hamiltonian cycle $C_1':=\langle w_1,P^1[w_1,z_1],z_1,w_1\rangle$ in $S_{n+1}^{2:1}-\{w,z\}$. The edge $(w_1,z_1)$ has a coupled pair-edge $(w_1',z_1')$ in $S_{n+1}^{2:b_1}$. According to Lemma~\ref{11}, $(w_1',z_1')$ lies on a cycle $C_1'':=\langle w_1',P^1[w_1',z_1'],z_1',w_1'\rangle$ in $S_{n+1,s-1}^2=S_{n+1}^{2}-S_{n+1}^{2:1}-S_{n+1,n-s+1}^{2}$. By concatenating $C_1'$ and $C_1''$, we obtain a cycle $C_1:=\langle w_1,P^1[w_1,z_1],z_1,z_1',P^1[z_1',w_1'],w_1',w_1\rangle$ of length $s\cdot n! - 2$ containing $u$. It follows that $C_1$ and $C_2$ are the desired $2$-DCC, as illustrated in Figure \ref{fig:17}.
     \begin{figure}[h!]
  \centering
  \includegraphics[width=\textwidth]{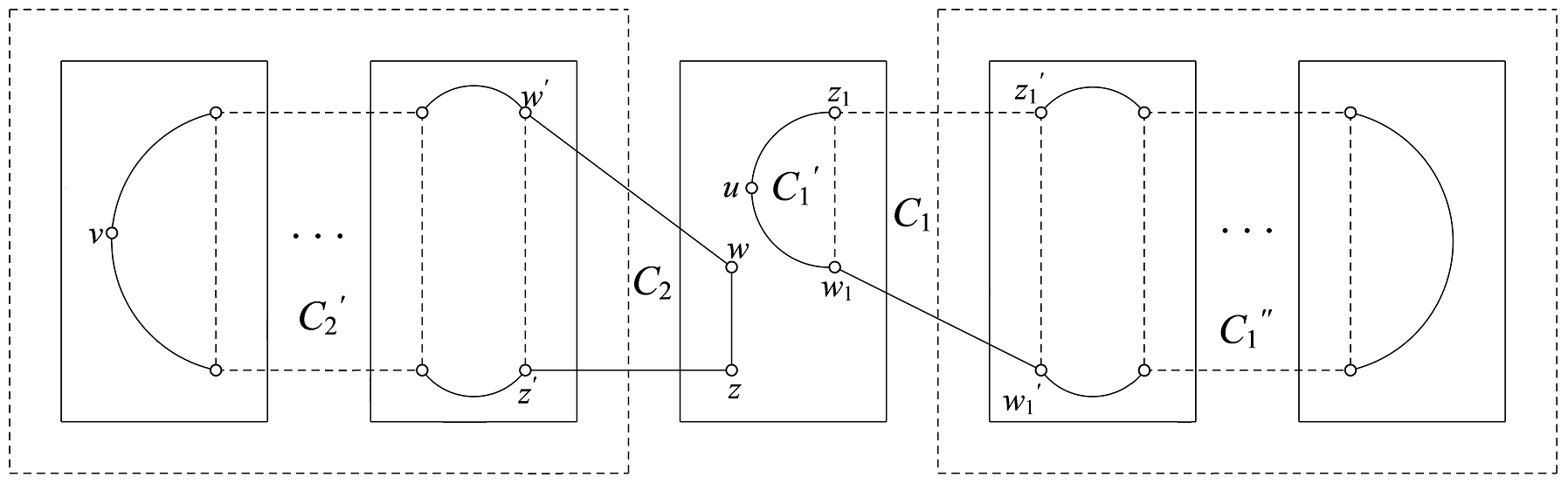}
     \makebox[0.2\textwidth][c]{$S_{n+1,n-s+1}^2$}
  \makebox[0.35\textwidth][c]{$S_{n+1}^{2:1}$}
  \makebox[0.2\textwidth][c]{$S_{n+1,s-1}^2$}
  \caption{Illustration of the proof of Case 2.2 in Theorem \ref{3-3}.}
  \label{fig:17}
  \end{figure}

\textbf{Case 2.3.} $\ell=s\cdot n!-1$.

Choose a vertex $u_1=ia_1a_2\cdots a_{n-1}1$ in $S_{n+1}^{2:1}$ such that $u_1\neq u$. According to \cite[Lemma~2.8 (i)]{Li25},~$S_{n+1}^{2:1}-\{u_1\}$ contains a Hamiltonian cycle $C_1'$ and $u\in V(C_1')$. If $s=1$, then $C_1=C_1'$. By Lemma \ref{6}, there exists a cycle $C_2$ such that $V(C_2)=V(S_{n+1}^{2}-S_{n+1}^{2:1})\cup\{u_1\}$. The length of $C_2$ is $n\cdot n!+1=(n+1)!-(n!-1)=(n+1)!-\ell$. Hence, $C_1$ and $C_2$ form the desired $2$-DCC.

Suppose $s\ge2$. Let $w=b_1b_2\cdots b_n1\in V(C_1')$, where $\{b_1,b_2\}\cap\{i,a_1\}=\emptyset$, and $w$ has a neighbor $w_1$ in $C_1':=\langle w,w_1,P^1[w_1,w],w\rangle$. Then $(w,w_1)$ has a coupled pair-edge $(w',w_1')$ in $S_{n+1}^{2:m_1}$, where $m_1=b_1$ or $b_2$. Choose $n-s-1$ integers $i_1,i_2,\dots,i_{n-s-1}\in[n+1]\setminus\{1,i,a_1,m_1\}$. According to Lemma \ref{6}, there exists a cycle $C_2$ in $S_{n+1}^2$ containing $v$ such that $V(C_2)=V(S_{n+1}^{a_1}\oplus S_{n+1}^{i}\oplus S_{n+1}^{i_1}\oplus\cdots\oplus S_{n+1}^{i_{n-s-1}})\cup \{u_1\}$. Then $|C_2|=(n-s+1)\cdot n!+1$. By Lemma \ref{11},~$(w',w_1')$ lies on a Hamiltonian cycle $C_1'':=\langle w',w_1',P^1[w_1',w'],w'\rangle$ of $S_{n+1,s-1}^2=S_{n+1}^{2:m_1}\oplus S_{n+1}^{2:j_1}\oplus S_{n+1}^{2:j_2}\oplus\cdots\oplus S_{n+1}^{2:j_{s-2}}$, where $j_1,j_2,\dots,j_{s-2}\in[n+1]\setminus\{i_1,i_2,\dots,i_{n-s-1},1,a_1,i,m_1\}$. Hence, $C_1'$ can be extended to a cycle $C_1:=\langle w_1,P^1[w_1,w],w,w',P^1[w',w_1'],w_1',w_1\rangle$ by concatenating with $C_1''$. Consequently, $|C_1|=(s-1)\cdot n!+n!-1=s\cdot n!-1$. Hence, $C_1$ and $C_2$ are vertex-disjoint, as illustrated in Figure \ref{fig:18}.
     \begin{figure}[h!]
  \centering
  \includegraphics[width=\textwidth]{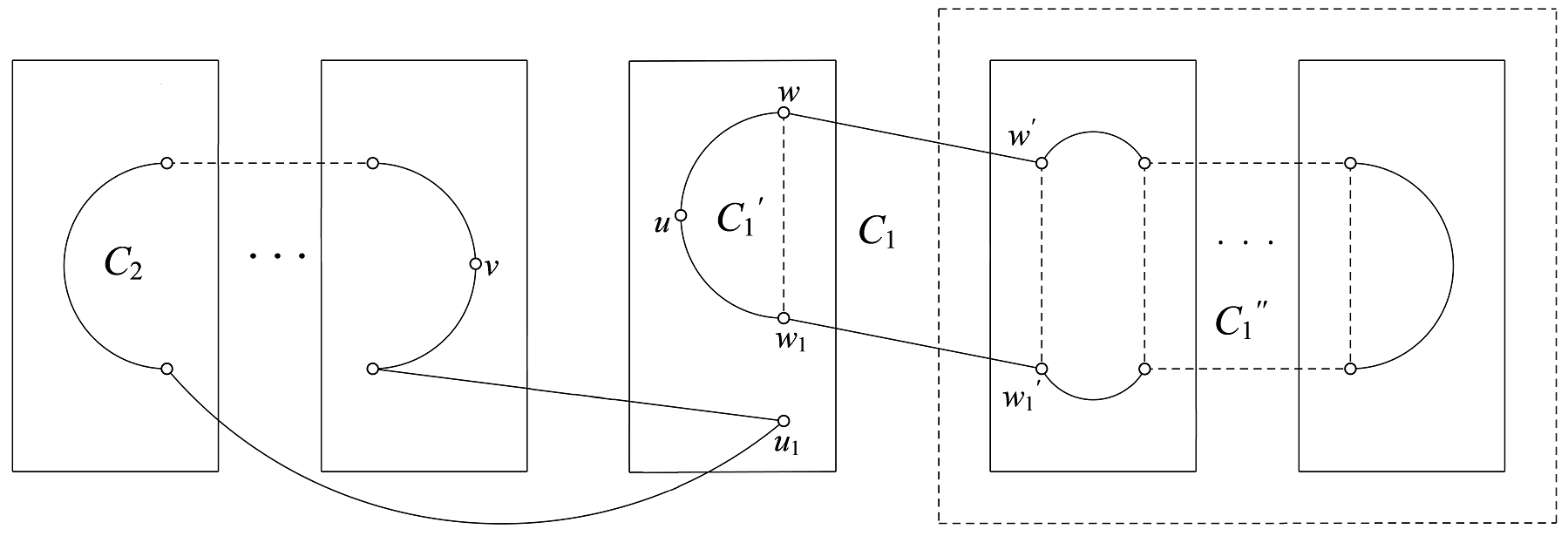}
  \makebox[0.18\textwidth][c]{$S_{n+1}^{2:a_1}$}
  \makebox[0.18\textwidth][c]{$S_{n+1}^{2:i}$}
  \makebox[0.2\textwidth][c]{$S_{n+1}^{2:1}$}
  \makebox[0.4\textwidth][c]{$S_{n+1,s-1}^{2}$}
  \caption{Illustration of the proof of Case 2.3 in Theorem \ref{3-3}.}
  \label{fig:18}
  \end{figure}

\textbf{Case 2.4.} $\ell=s\cdot n!$.

Choose a vertex $u_1=u\circ(1,2)\in S_{n+1}^{2:1}$. According to Lemma~\ref{11}, for integers $i_1,i_2,\dots,i_{s-1}\in[n+1]\setminus\{1,i\}$, there exists a cycle $C_1$ in $S_{n+1,s}^2 = S_{n+1}^{2:1} \oplus S_{n+1}^{2:i_1} \oplus \cdots \oplus S_{n+1}^{2:i_{s-1}}$ containing $u$. The length of $C_1$ is $n! + (s-1)\cdot n! = s\cdot n!$. Similarly, we can construct a cycle $C_2$ in $S_{n+1,n+1-s}^2=S_{n+1}^2 - S_{n+1,s}^2$ satisfying $v\in V(C_2)$ and $|C_2| =(n+1-s)\cdot n! = (n+1)! - \ell$. Hence, $C_1$ and $C_2$ form the desired $2$-DCC, as illustrated in Figure \ref{fig:19}.
     \begin{figure}[h!]
  \centering
  \includegraphics[width=\textwidth]{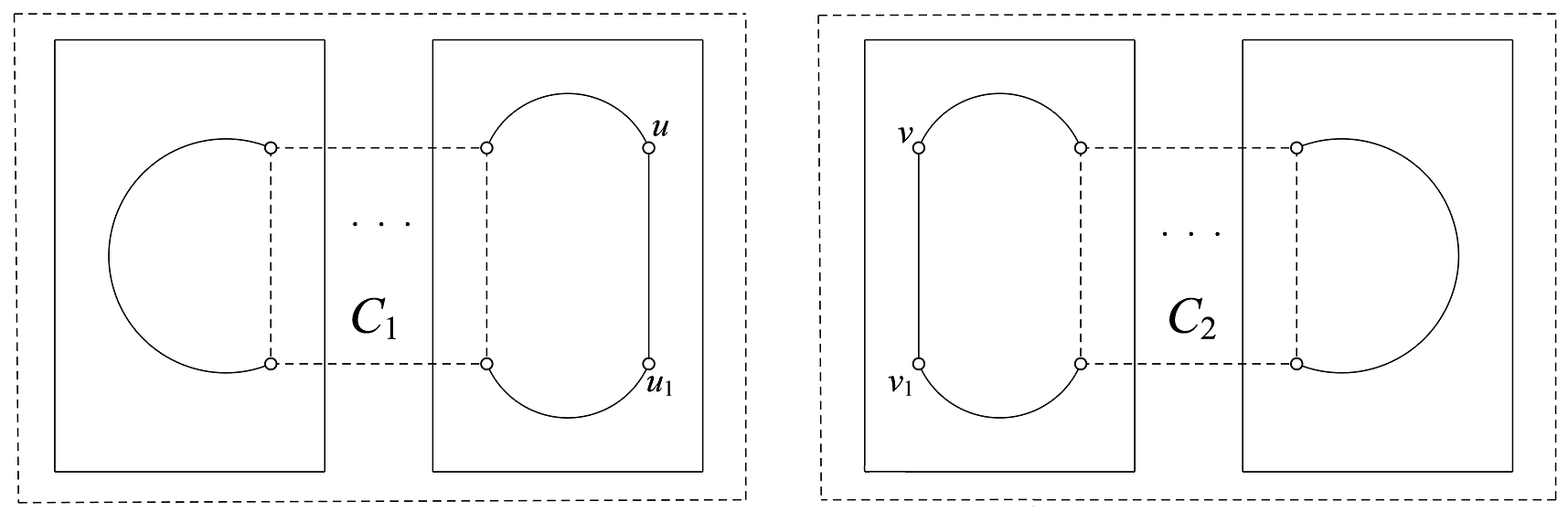}
        \makebox[0.45\textwidth][c]{$S_{n+1,s}^{2}$}
  \makebox[0.45\textwidth][c]{$S_{n+1,n-s+1}^{2}$}
  \caption{Illustration of the proof of Case 2.4 in Theorem \ref{3-3}.}
  \label{fig:19}
  \end{figure}

%

\textbf{Case 2.5.} $\ell=s\cdot n!+1$.

If $s = 1$, then $\ell = n! + 1$. By~\cite[Lemma 2.8 (i)]{Li25}, $S_{n+1}^{2:1} - \{u'\}$ admits a Hamiltonian cycle $C_1'$ containing $u$, where $u'=a_1a_2\cdots a_n1$ with $a_1=i$. Let $w = a_3a_4\cdots a_1a_21\in V(C_1')$, where $a_3,a_4\in [n+1]\setminus\{a_1,a_2\}$. Let $w_1$ be a neighbor of $w$ on $C_1':=\langle w,w_1,P^1[w_1,w],w\rangle$. The edge $(w,w_1)$ has a coupled pair-edge $(w', w_1')$ in $S_{n+1}^{2:m}$, where $m = a_3$ or $a_4$. Let
\[
C_1 := \langle w', w_1', w_1, P^1[w_1, w], w, w' \rangle.
\]
Then $|C_1|=n!+1$.

By \cite[Lemma 2.8 (ii) and (iii)]{Li25}, $S_{n+1}^{2:m} - \{w', w_1'\}$ contains a Hamiltonian cycle $C_2' = \langle z, z_1, P^2[z_1,z], z \rangle$, where $z = a_1 a_j \cdots 1 m$ with $j \in \{3, 4\} \setminus \{m\}$, and $z_1$ is a neighbor of $z$ in $C_2'$.
Regardless of whether $(z, z_1)$ is a $(1,2)$-edge or an $i$-edge, the vertex $z_1$ has a neighbor in either $S_{n+1}^{2:a_j}$ or $S_{n+1}^{2:a_1}$, while $z$ has a neighbor in the other subgraph. Since the proofs are similar for $z_1=z\circ(1,2),zs_{k}^+$ and $zs_{k}^-$, we may assume $z_1=z\circ(1,2)=a_ja_1\cdots1m$. Then we let
\[
\begin{aligned}
u_1 &= z_1s_{n+1}^+ = m a_j \cdots 1 a_1, &\quad
u_j &= zs_{n+1}^+ =  m a_1\cdots 1 a_j,\\
v_1 &= u' s_{n+1}^- = a_2 1 \cdots a_n a_1, &\quad
v_2 &= u' s_{n+1}^+ = 1 a_1 \cdots a_n a_2.
\end{aligned}
\]

If $n=4$, then set $P[v_j,u_2]=(v_j,u_2)$, where $u_2 = a_j \cdots a_2$ and $v_j = u_2s_{n+1}^-$. Now suppose $n\ge 5$. For each $r \in [5, n-1]$, denote $u_r = a_{r+1} \cdots a_2$ and $v_{r+1} = u_rs_{n+1}^-$.
Let $u_2=a_5 \cdots a_2$, $u_n = a_j \cdots a_n$, $v_5 = u_2s_{n+1}^-$ and $v_j =u_ns_{n+1}^-$. By Lemma \ref{2}, there exists a Hamiltonian path $P[u_r,v_r]$ in $S_{n+1}^{2:a_r}$. Denote $P[v_j,u_2]=\langle v_j,u_n,P[u_n,v_n],v_n,\ldots,u_5,P[u_5,v_5],v_5,u_2\rangle$. We conclude that $C_1$ and
\[
\begin{aligned}
C_2:= \langle \; & u',v_1,P[v_1,u_1],u_1,z_1,P^2[z_1,z],z,u_j,P[u_j,v_j],v_j,P[v_j,u_2],u_2,P[u_2,v_2],v_2,u' \rangle
\end{aligned}
\]
form the desired $2$-DCC for $n\geq4$, as illustrated in Figure \ref{fig:20}.
     \begin{figure}[h!]
  \centering
  \includegraphics[width=\textwidth]{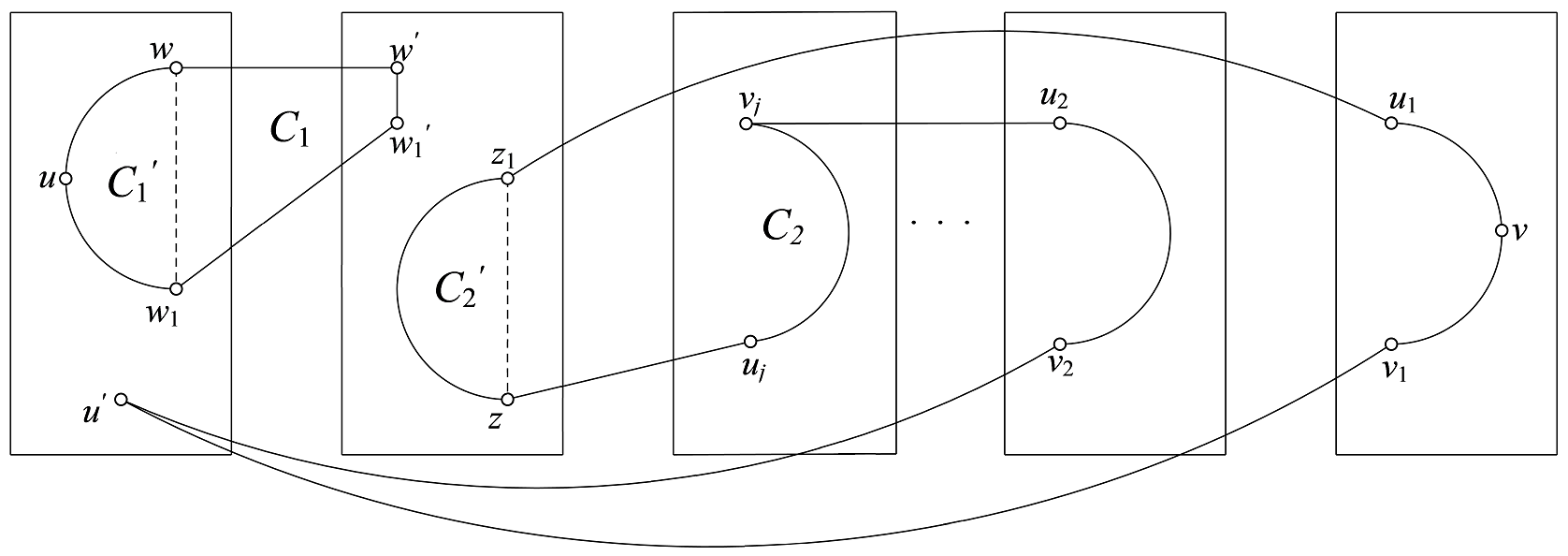}
    \makebox[0.18\textwidth][c]{$S_{n+1}^{2:1}$}
  \makebox[0.2\textwidth][c]{$S_{n+1}^{2:m}$}
  \makebox[0.2\textwidth][c]{$S_{n+1}^{2:a_j}$}
  \makebox[0.2\textwidth][c]{$S_{n+1}^{2:a_2}$}
  \makebox[0.18\textwidth][c]{$S_{n+1}^{2:i}$}
  \caption{Illustration of the proof of Case 2.5 in Theorem \ref{3-3}.}
  \label{fig:20}
  \end{figure}

For $s \ge 2$, we choose a vertex $w=1a_n\cdots a_2a_1$ in $S_{n+1}^{2:a_1}$ such that $a_n\neq i$. According to Lemma~\ref{6}, there exists a cycle $C_1$ in $S_{n+1}^{2}$ such that $V(C_1)=V(S_{n+1}^{2:1}\oplus S_{n+1}^{2:a_n}\oplus S_{n+1}^{2:i_1}\oplus S_{n+1}^{2:i_2}\oplus\cdots\oplus S_{n+1}^{2:i_{s-2}})\cup \{w\}$,
where $\{i_1, \dots, i_{s-2}\} \subset [n+1] \setminus \{1,a_1,i\}$ and $|C_1|=s\cdot n!+1$.
By~\cite[Lemma 2.8 (i)]{Li25}, $S_{n+1}^{2:a_1} - \{w\}$ admits a Hamiltonian cycle $C_2':=\langle z,z_1,P^2[z_1,z],z\rangle$, where $z=ia_r\cdots a_1$ with $a_r \in [n+1] \setminus \{1,i_1, \dots, i_{s-2}\}$, and $z_1$ is a neighbor of $z$ in $C_2'$. By Lemma~\ref{7}, there exists a cycle $C_2'':=\langle z,z_1,P^2[z_1,z],z\rangle$ such that $V(C_2'')=V(S_{n+1}^{2:i} \oplus S_{n+1}^{2:a_r} \oplus S_{n+1}^{2:j_1} \oplus \cdots \oplus S_{n+1}^{2:j_{n-s-2}})\cup \{z,z_1\}$ with $\{j_1, \dots, j_{n-s-2}\} \subseteq [n+1] \setminus \{1, i, a_r, a_1, i_1, \dots, i_{s-2}\}$.
Finally, let
\[
C_2 := \langle z_1,P^2[z_1,z],z,P^2[z,z_1],z_1 \rangle.
\]
Then $C_2$ contains $v$ and has length $(n-s) \cdot n! + n! - 1 = (n+1)! - \ell$.
Hence, $C_1$ and $C_2$ form the desired $2$-DCC, as illustrated in Figure \ref{fig:21}.
     \begin{figure}[h!]
  \centering
  \includegraphics[width=\textwidth]{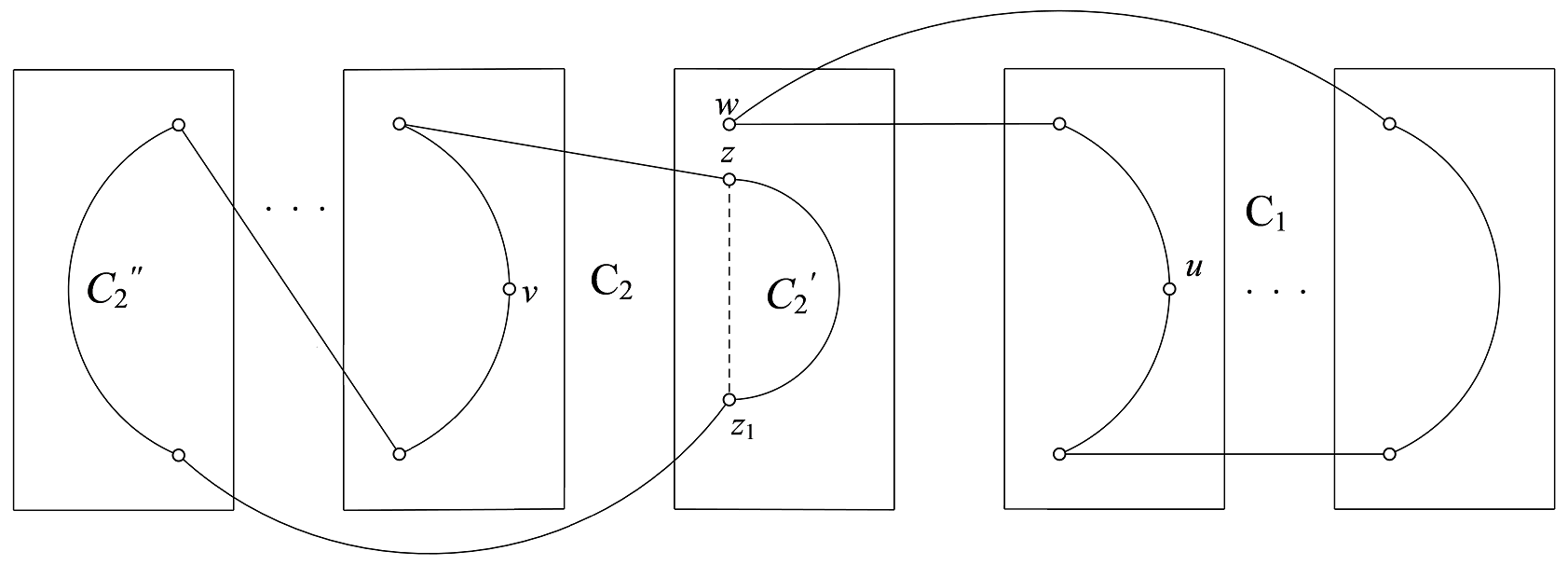}
    \makebox[0.18\textwidth][c]{$S_{n+1}^{2:a_r}$}
  \makebox[0.2\textwidth][c]{$S_{n+1}^{2:i}$}
  \makebox[0.2\textwidth][c]{$S_{n+1}^{2:a_1}$}
  \makebox[0.2\textwidth][c]{$S_{n+1}^{2:1}$}
  \makebox[0.18\textwidth][c]{$S_{n+1}^{2:a_n}$}
  \caption{Illustration of the proof of Case 2.5 in Theorem \ref{3-3}.}
  \label{fig:21}
  \end{figure}

\textbf{Case 2.6.} $\ell=s\cdot n!+2$.

Choose two adjacent vertices $w=a_1a_2\cdots a_n1$ and $w_1=w\circ(1,2)$ in $S_{n+1}^{2:1}$, where $a_1\ne i$. By Lemma~\ref{11}, for integers $i_1,i_2,\dots,i_{s-1}\in[n+1]\setminus\{1,i,a_1\}$, there exists a Hamiltonian cycle $C_1':=\langle w,w_1,P^1[w_1,w],w\rangle$ in $S_{n+1,s}^2 = S_{n+1}^{2:1} \oplus S_{n+1}^{2:i_1} \oplus \cdots \oplus S_{n+1}^{2:i_{s-1}}$. The edge $(w,w_1)$ has a coupled pair-edge $(w',w_1')$ in $S_{n+1}^{2:a_1}$. Hence, $C_1'$ can be extended to a cycle $C_1:=\langle w_1,P^1[w_1,w],w,w',w_1',w_1\rangle$ and $|C_1|=s\cdot n!+2$.

By Lemma~\ref{3-2}, $S_{n+1}^{2:a_1}-\{w',w_1'\}$ has a Hamiltonian cycle $C_2':=\langle z,z_1,P^2[z_1,z],z\rangle$, where $z=ib_1\cdots b_{n-1}a_1$ and $z_1=z\circ(1,2)$. The edge $(z,z_1)$ has a coupled pair-edge $(z',z_1')$ in $S_{n+1}^{2:i}$. According to Lemma~\ref{11}, there exists a Hamiltonian cycle $C_2'':=\langle z',z_1',P^2[z_1',z'],z'\rangle$ in $S_{n+1,n-s}^2=S_{n+1}^2- S_{n+1,s}^2-S_{n+1}^{2:a_1}$. Let \[C_2:=\langle z_1,P^2[z_1,z],z,z',P^2[z',z_1'],z_1',z_1\rangle.\] Then $C_2$ contains $v$ and has length $(n-s) \cdot n! + n! - 2 = (n+1)! - \ell$.
Thus, $C_1$ and $C_2$ form the desired $2$-DCC, as illustrated in Figure \ref{fig:22}.
     \begin{figure}[h!]
  \centering
  \includegraphics[width=\textwidth]{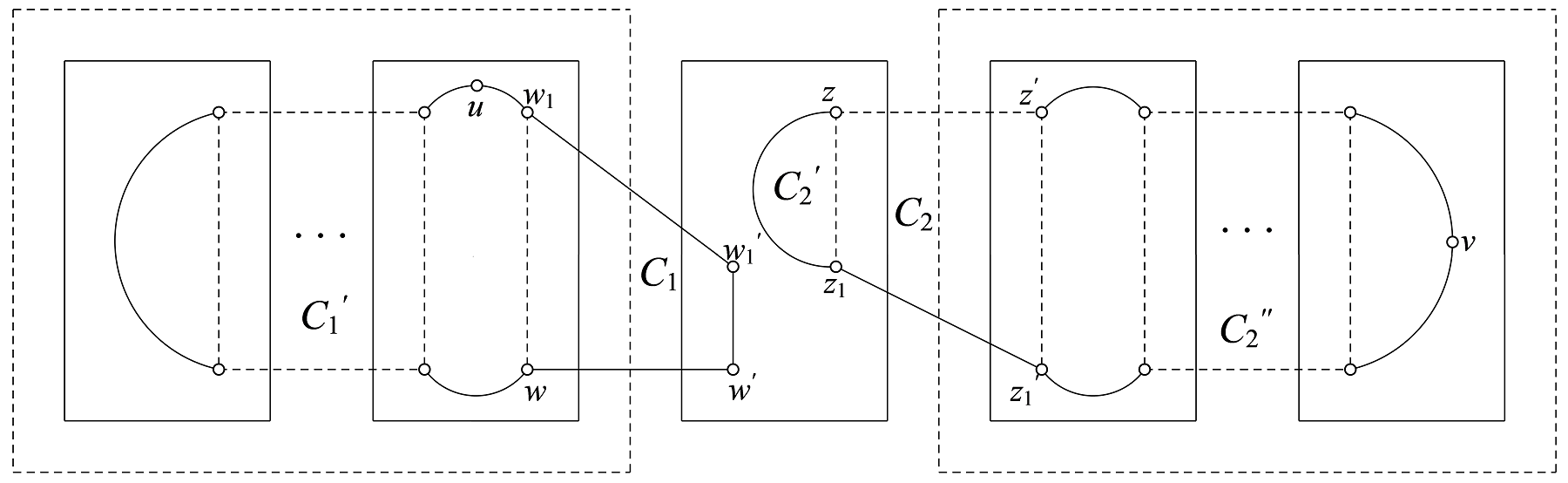}
   \makebox[0.2\textwidth][c]{$S_{n+1,s}^2$}
  \makebox[0.35\textwidth][c]{$S_{n+1}^{2:a_1}$}
  \makebox[0.2\textwidth][c]{$S_{n+1,n-s}^2$}
  \caption{Illustration of the proof of Case 2.6 in Theorem \ref{3-3}.}
  \label{fig:22}
  \end{figure}

  Therefore, for any pair of distinct vertices $u,v$ and any integer $\ell$ satisfying $3\le \ell \le \frac{(n+1)!}{2}$, there exist two vertex-disjoint cycles satisfying the prescribed conditions. Thus the induction step is finished, and the theorem follows.
\end{proof}

%
%
%

\appendix
\begin{table}[H]
\centering
\begin{adjustbox}{angle=90,max height=\textheight} 
\begin{minipage}{\textheight} 
\centering
\section{Additional Tables}
\caption{Some disjoint cycles in $S_4^2$ containing $u=1234$ and $v=2134$ respectively.}
\label{3-11}

\footnotesize
\setlength{\tabcolsep}{2pt}
\renewcommand{\arraystretch}{0.9}

\begin{tabular}{@{}c l@{}}
\toprule
$\ell$ & Cycles \\
\midrule

$3$ &
$C_1=\langle\mathbf{1234},3124,2314,1234\rangle$\\
&
$C_2=\langle\mathbf{2134},1432,4132,1342,3142,4312,3412,4213,2143,1243,2413,4123,1423,4321,3241,2431,4231,2341,3421,1324,3214,2134\rangle$
\\
\midrule

$4$ &
$C_1=\langle\mathbf{1234},4132,3412,2314,1234\rangle$\\
&
$C_2=\langle\mathbf{2134},3214,1324,3124,1423,4123,1243,2143,4213,2413,4312,1432,3142,1342,3241,2431,4321,3421,2341,4231,2134\rangle$
\\
\midrule

$5$ &
$C_1=\langle\mathbf{1234},2314,3214,1324,3124,1234\rangle$\\
&
$C_2=\langle\mathbf{2134},1432,4132,1342,3142,4312,3412,4213,2413,1243,4123,1423,2143,3241,2341,3421,4321,2431,4231,2134\rangle$
\\
\midrule

$6$ &
$C_1=\langle\mathbf{1234},3124,2314,3412,1342,4132,1234\rangle$\\
&
$C_2=\langle\mathbf{2134},1324,3214,4312,1432,3142,1243,4123,2413,4213,1423,2143,3241,2341,3421,4321,2431,4231,2134\rangle$
\\\midrule

$7$ &
$C_1=\langle\mathbf{1234},4132,3412,2314,3214,1324,3124,1234\rangle$\\
&
$C_2=\langle\mathbf{2134},1432,4312,3142,1342,2143,1243,4123,2413,4213,1423,4321,3421,2341,3241,2431,4231,2134\rangle$
\\\midrule

$8$ &
$C_1=\langle\mathbf{1234},4132,1342,3412,2314,3214,1324,3124,1234\rangle$\\
&
$C_2=\langle\mathbf{2134},1432,4312,3142,1243,2143,4213,2413,4123,1423,4321,3421,2341,3241,2431,4231,2134\rangle$
\\\midrule

$9$ &
$C_1=\langle\mathbf{1234},3124,2314,3412,4312,1432,3142,1342,4132,1234\rangle$\\
&
$C_2=\langle\mathbf{2134},1324,3214,2413,4123,1243,2143,4213,1423,4321,3421,2341,3241,2431,4231,2134\rangle$
\\\midrule

$10$ &
$C_1=\langle\mathbf{1234},3124,2314,4213,3412,4312,1432,3142,1342,4132,1234\rangle$\\
&
$C_2=\langle\mathbf{2134},1324,3214,2413,4123,1243,2143,1423,4321,3421,2341,3241,2431,4231,2134\rangle$
\\\midrule

$11$ &
$C_1=\langle\mathbf{1234},3124,2314,4213,2143,1342,3412,4312,3142,1432,4132,1234\rangle$\\
&
$C_2=\langle\mathbf{2134},1324,3214,2413,1243,4123,1423,4321,3421,2341,3241,2431,4231,2134\rangle$
\\\midrule

$12$ &
$C_1=\langle\mathbf{1234},3124,2314,4213,1423,2143,1342,3412,4312,3142,1432,4132,1234\rangle$\\
&
$C_2=\langle\mathbf{2134},1324,3214,2413,1243,4123,3421,4321,2431,3241,2341,4231,2134\rangle$
\\

\bottomrule
\end{tabular}

\end{minipage} 
\end{adjustbox}
\end{table}

\begin{table}[htbp]
\centering
\begin{adjustbox}{angle=90,max height=\textheight} 
\begin{minipage}{\textheight} 
\centering
\caption{Some disjoint cycles in $S_4^2$ containing $u=1234$ and four vertex $\{3124,1324,3214,2314\}$ respectively.}
\label{3-12}

\footnotesize
\setlength{\tabcolsep}{2pt}
\renewcommand{\arraystretch}{0.9}

\begin{tabular}{@{}c l@{}}
\toprule
$\ell$ & Cycles \\
\midrule

$3$ &
$C_1=\langle\mathbf{1234},4132,2431,1234\rangle$\\
&
$C_2=\langle\mathbf{3124,1324},2134,\mathbf{3214,2314},3412,1342,3142,1432,4312,2413,4213,1423,4123,1243,2143,3241,2341,4231,3421,4321,3124\rangle$
\\
\midrule

$4$ &
$C_1=\langle\mathbf{1234},4132,1432,2134,1234\rangle$\\
&
$C_2=\langle\mathbf{3124,1324,3214,2314},3412,1342,3142,4312,2413,4213,1423,4123,1243,2143,3241,2341,3421,4231,2431,4321,3124\rangle$
\\
\midrule

$5$ &
$C_1=\langle\mathbf{1234},4132,2431,4231,2134,1234\rangle$\\
&
$C_2=\langle\mathbf{3124,1324,3214,2314},3412,1342,3142,1432,4312,2413,4213,1423,4123,1243,2143,3241,2341,3421,4321,3124\rangle$
\\
\midrule

$6$ &
$C_1=\langle\mathbf{1234},4132,1342,3142,1432,2134,1234\rangle$\\
&
$C_2=\langle\mathbf{3124,1324,3214,2314},3412,4312,2413,4213,1423,4123,1243,2143,3241,2341,3421,4231,2431,4321,3124\rangle$
\\\midrule

$7$ &
$C_1=\langle\mathbf{1234},4132,1342,3142,4312,1432,2134,1234\rangle$\\
&
$C_2=\langle\mathbf{3124,1324,3214,2314},3412,4213,2413,1243,4123,1423,2143,3241,2341,3421,4231,2431,4321,3124\rangle$
\\\midrule

$8$ &
$C_1=\langle\mathbf{1234},4132,1342,3412,4312,3142,1432,2134,1234\rangle$\\
&
$C_2=\langle\mathbf{3124,1324,3214,2314},4213,2413,1243,4123,1423,2143,3241,2341,3421,4231,2431,4321,3124\rangle$
\\\midrule

$9$ &
$C_1=\langle\mathbf{1234},2431,4132,1342,3412,4312,3142,1432,2134,1234\rangle$\\
&
$C_2=\langle\mathbf{3124,1324,3214,2314},4213,2413,1243,4123,1423,2143,3241,2341,4231,3421,4321,3124\rangle$
\\\midrule

$10$ &
$C_1=\langle\mathbf{1234},4132,1342,2143,4213,3412,4312,3142,1432,2134,1234\rangle$\\
&
$C_2=\langle\mathbf{1324,3124,2314,3214},2413,1243,4123,1423,4321,3241,2431,4231,2341,3421,1324\rangle$
\\\midrule

$11$ &
$C_1=\langle\mathbf{1234},4132,1342,2143,1423,4213,3412,4312,3142,1432,2134,1234\rangle$\\
&
$C_2=\langle\mathbf{1324,3124,2314,3214},2413,4123,1243,2341,4231,2431,3241,4321,3421,1324\rangle$
\\\midrule

$12$ &
$C_1=\langle\mathbf{1234},4132,1342,2143,1243,2413,4213,3412,4312,3142,1432,2134,1234\rangle$\\
&
$C_2=\langle\mathbf{3124,2314,3214,1324},4123,3421,4231,2341,3241,2431,4321,1423,3124\rangle$
\\

\bottomrule
\end{tabular}

\end{minipage} 
\end{adjustbox}
\end{table}

\begin{table}[htbp]
\centering
\begin{adjustbox}{angle=90,max height=\textheight} 
\begin{minipage}{\textheight} 
\centering
\caption{Some disjoint cycles in $S_4^2$ containing $u=1234$ and $V(S_4^{2:2})$ respectively.}
\label{3-13}

\footnotesize
\setlength{\tabcolsep}{2pt}
\renewcommand{\arraystretch}{0.9}

\begin{tabular}{@{}c l@{}}
\toprule
$\ell$ & Cycles \\
\midrule

$3$ &
$C_1=\langle\mathbf{1234},3124,2314,1234\rangle$\\
&
$C_2=\langle2134,\mathbf{1432,4132,1342,3142,4312,3412},4213,2143,1243,2413,4123,1423,4321,3241,2431,4231,2341,3421,1324,3214,2134\rangle$\\
\midrule

$4$ &
$C_1=\langle\mathbf{1234},2134,1324,3124,1234\rangle$\\
&
$C_2=\langle\mathbf{4132,1432,3142,1342,3412,4312},3214,2314,4213,2413,1243,4123,1423,2143,3241,2341,4231,3421,4321,2431,4132\rangle$\\

$5$ &
$C_1=\langle\mathbf{1234},2134,3214,1324,3124,1234\rangle$\\
&
$C_2=\langle\mathbf{4132,1432,4312,3142,1342,3412},2314,4213,2413,1243,4123,1423,2143,3241,2341,4231,3421,4321,2431,4132\rangle$
\\
\midrule

$6$ &
$C_1=\langle\mathbf{1234},2134,1324,3214,2314,3124,1234\rangle$\\
&
$C_2=\langle\mathbf{4132,1432,4312,3142,1342,3412},4213,2413,1243,4123,1423,2143,3241,2341,4231,3421,4321,2431,4132\rangle$
\\
\midrule

$7$ &
$C_1=\langle\mathbf{1234},2134,3214,1324,3124,4321,2431,1234\rangle$\\
&
$C_2=\langle\mathbf{1432,4132,1342,3142,4312,3412},2314,4213,2413,1243,4123,1423,2143,3241,2341,3421,4231,1432\rangle$
\\
\midrule

$8$ &
$C_1=\langle\mathbf{1234},2314,3214,2134,1324,3124,4321,2431,1234\rangle$\\
&
$C_2=\langle\mathbf{1432,4132,1342,3142,4312,3412},4213,2413,1243,4123,1423,2143,3241,2341,3421,4231,1432\rangle$
\\
\midrule

$9$ &
$C_1=\langle\mathbf{1234},2314,3214,2134,1324,3124,4321,3241,2431,1234\rangle$\\
&
$C_2=\langle\mathbf{1432,4132,1342,3142,4312,3412},4213,2413,4123,1423,2143,1243,2341,3421,4231,1432\rangle$
\\
\midrule

$10$ &
$C_1=\langle\mathbf{1234},2134,1324,3214,2314,4213,2413,4123,1423,3124,1234\rangle$\\
&
$C_2=\langle\mathbf{1432,4132,1342,3412,4312,3142},1243,2143,3241,2341,3421,4321,2431,4231,1432\rangle$
\\
\midrule

$11$ &
$C_1=\langle\mathbf{1234},2134,1324,3214,2314,4213,2413,1243,4123,1423,3124,1234\rangle$\\
&
$C_2=\langle\mathbf{1432,4132,3412,4312,3142,1342},2143,3241,2341,3421,4321,2431,4231,1432\rangle$
\\
\midrule

$12$ &
$C_1=\langle\mathbf{1234},2134,1324,3214,2314,4213,2143,1243,2413,4123,1423,3124,1234\rangle$\\
&
$C_2=\langle\mathbf{1432,4132,3412,4312,3142,1342},3241,2341,3421,4321,2431,4231,1432\rangle$
\\

\bottomrule
\end{tabular}

\end{minipage} 
\end{adjustbox}
\end{table}


\newpage
	\begin{thebibliography}{00}
		\bibitem{Li25} H. Li, L. Chen and M. Lu, Two-disjoint-cycle-cover pancyclicity of split-star networks,
		\textit{Applied Mathematics and Computation}, 487 (2025), 129085.

        \bibitem{CLP01} E. Cheng, M. J. Lipman and H. Park, Super connectivity of star graphs, alternating group graphs and split-stars,
        \textit{Ars Combinatoria}, 59 (2001), 107--116.

        \bibitem{CL00} E. Cheng and M. J. Lipman, Orienting split-stars and alternating group graphs,
        \textit{Networks: An International Journal}, 35 (2000), 139--144.


        \bibitem{ZH18} S. L. Zhao and R. X. Hao, The generalized connectivity of some regular graphs,
        \textit{arXiv preprint} arXiv:1808.10074 (2018).

        \bibitem{GHC19} M. M. Gu, R. X. Hao and J. M. Chang, Measuring the vulnerability of alternating group graphs and split-star networks in terms of component connectivity,
        \textit{IEEE Access}, 7 (2019), 97745--97759.

        \bibitem{GHC21} M. M. Gu, R. X. Hao and J. M. Chang, Reliability analysis of alternating group graphs and split-stars,
        \textit{The Computer Journal}, 64 (2021), 1425--1436.

        \bibitem{LXZH15} L. M. Lin, L. Xu, S. M. Zhou and S. Y. Hsieh, The extra, restricted connectivity and conditional diagnosability of split-star networks,
        \textit{IEEE Transactions on Parallel and Distributed Systems}, 27 (2015), 533--545.

        \bibitem{LHWX20} L. M. Lin, Y. Z. Huang, X. D. Wang and L. Xu, Restricted connectivity and good-neighbor diagnosability of split-star networks,
        \textit{Theoretical Computer Science}, 824 (2020), 81--91.

        \bibitem{ZW23} L. N. Zhao and S. Y. Wang, Structure connectivity and substructure connectivity of split-star networks,
        \textit{Discrete Applied Mathematics}, 341 (2023), 359--371.

        \bibitem{QMS22} H. Qiao, J. Meng and E. Sabir, Embedding spanning disjoint cycles in enhanced hypercube networks with prescribed vertices in each cycle,
        \textit{Applied Mathematics and Computation}, 435 (2022), 127481.

        \bibitem{NXL21} R. Niu, M. Xu and H. J. Lai, Two-disjoint-cycle-cover vertex bipancyclicity of the bipartite generalized hypercube,
        \textit{Applied Mathematics and Computation}, 400 (2021), 126090.


        \bibitem{WHC20} C. Wei, R. X. Hao and J. M. Chang, Two-disjoint-cycle-cover bipancyclicity of balanced hypercubes,
        \textit{Applied Mathematics and Computation}, 381 (2020), 125305.

        \bibitem{BM08} J. A. Bondy and U. S. R. Murty,
        \textit{Graph Theory},Springer, New~York, 2008.

        \bibitem{KC15} T. L. Kung and H. C. Chen, Complete cycle embedding in crossed cubes with two-disjoint-cycle-cover pancyclicity,
        \textit{IEICE Transactions on Fundamentals of Electronics, Communications and Computer Sciences}, E98~(12) (2015), 2670--2676.

        \bibitem{LLC21} J. Li, X. J. Li and E. Cheng, Super spanning connectivity of split-star networks,
        \textit{Information Processing Letters}, 166 (2021), 106037.















		
	\end{thebibliography}
\end{document}